\documentclass{amsart}
\usepackage{graphicx}
\usepackage{amsmath}
\usepackage{amssymb}%
\usepackage{amsfonts}%
\usepackage{graphicx}
\usepackage[all]{xy}
\usepackage[mathscr]{euscript}
\usepackage{MnSymbol}

\newtheorem{theorem}{Theorem}[section]
\newtheorem{lemma}[theorem]{Lemma}

\theoremstyle{definition}
\newtheorem{definition}[theorem]{Definition}

\newtheorem{proposition}[theorem]{Proposition}
\newtheorem{corollary}[theorem]{Corollary}

\usepackage[all]{xy}
\usepackage{graphicx}
\theoremstyle{remark}
\newtheorem{remark}[theorem]{Remark}

\numberwithin{equation}{section}
\usepackage{graphicx}

\begin{document}
\title{Topological dynamics on finite directed graphs}
\date{\today}
\author{Jos\'{e} Ayala}
\address{FIA, Universidad Arturo Prat, Iquique, Chile}
\email{jayalhoff@gmail.com}
\author{Wolfgang Kliemann}
\address{Department of Mathematics, Iowa State University \\Ames, Iowa 50011, U.S.A.}
\email{kliemann@iastate.edu}
\subjclass[2010]{Primary 54H20, 05C20, 60J10; Secondary 37C70, 60J22 }
\keywords{Morse decomposition, directed graphs, Markov chians}
\maketitle

\begin{abstract} In this work we establish that finite directed graphs give rise to semiflows on the power set of their nodes. We analyze the topological dynamics for semiflows on finite directed graphs by characterizing Morse decompositions, recurrence behavior and attractor-repeller pairs under weaker assumptions. As is expected, the discrete metric plays an important role in our constructions and their consequences. The connections between the semiflow, graph theory and Markov chains are here explored. We lay the foundation for a dynamical systems
approach to hybrid systems with Markov chain type perturbations.
\end{abstract}


\section{Introduction}

The mathematical theory of dynamical systems analyzes, from an
axiomatic point of view, the common features of many models that
describe the behavior of systems in time. In its abstract form, a
dynamical system is given by a time set ${T}$ (with semigroup
operation $\circ$), a state space $M$, and a map
$\Phi:{T}\times M\rightarrow M$ that satisfies (i)
$\Phi(0,x)=x$ for all $x\in M$, describing the initial value, and
(ii) $\Phi(t\circ s,x)=\Phi(t,\Phi(s,x))$ for all $t,s\in{T}$
and $x\in M$. At the heart of the theory of dynamical systems is the
study of systems behavior when $t\rightarrow\pm\infty$ (qualitative behaviour), as well the change
in behaviour under variation of parameters (bifurcation theory) \cite{AH06, CONLEY78, Devaney, KH95, Robinson}.

In this work we consider dynamical systems on finite directed graphs without multiple edges. We analyze
their communication structure, i.e., equivalence classes of vertices
that can be reached mutually via sequences of edges. This leads to the set of {\it communicating classes} $\mathcal{C}$ of a graph and a reachability order $\preceq$ on $\mathcal{C}$. The key concept is that of an $L-$graph corresponding to graphs
for which each vertex has out-degree $\geq1$. As it turns out, these
are exactly the graphs for which the $\omega-$limit sets of the
associated semiflow are nonempty. To each graph $G=(V,E)$, where $V$ is the set of vertices, $\mathcal{P}(V)$ the power set of $V$,
and $E\subset V\times V$ the set of edges, we associate a semiflow
$\Phi
_{G}:\mathbb{N}\times\mathcal{P}(V)\rightarrow\mathcal{P}(V)$.
This semiflow is studied from the point of view of qualitative
behavior of dynamical systems. We adapt the concepts of
$\omega-$limit sets, (positive) invariance, recurrence, Morse
decompositions, attractors and attractor-repeller pairs to
$\Phi_{G}$ and prove characterizations equivalent to those in
\cite{AH06}. As it turns out, the finest Morse decomposition of the semiflow 
$\Phi_{G}$ corresponds to the decomposition of $G$ into the communicating classes $\mathcal C$. In addition, the order on the
communicating classes is equivalent to the order that accompanies a
Morse decomposition. Moreover, the connected components of the
recurrent set of $\Phi_{G}$ are exactly the (finest) Morse sets of
$\Phi_{G}$, i.e., the communicating classes of $G$. 

\noindent Graphs $G=(V,E)$ (and certain aspects of Markov chains)
are often studied using the adjacency matrix $A_{G}$. The products of $A_{G}$ describe the paths, and hence the
communication structure of $G$. We construct a semiflow
$\Psi_{A}:\mathbb{N}\times\mathcal{Q}^{d}\rightarrow\mathcal{Q}^{d}$
(where $\mathcal{Q}^{d}$ is the vertex set of the unit cube in $\mathbb{R}%
^{d}$) that is equivalent to the semiflow $\Phi_{G}$ defined on $\mathcal{P}%
(V)$, using Boolean matrix multiplication. This point of view is
somewhat different from the standard approach that uses regular
matrix multiplication and that does not lead to an equivalent
semiflow. The equivalence allows us to interpret all results
obtained for $\Phi_{G}$ in terms of certain linear iterated function
systems.

We subsequently apply the results obtained
for graphs and their semiflows to the study of general finite Markov
chains. Our results are presented in the form of a ``3-language
dictionary'': Each key concept for Markov chains is ``translated'' into
graph language and into semiflow language. This dictionary is
contained in Facts 1-13 and Fact 15. Note that our concepts and
results in terms of global dynamics only deal with the communication
structure of graphs (or the qualitative behavior of semiflows) and
hence they do not contain the probabilistic information of a Markov
chain. But, it turns out that a simple result on the geometric decay
of certain probabilities (Lemma \ref{leaking}) is sufficient
to recapture all the relevant probabilistic information. Facts 14
and 16-18 describe the long term behavior of general Markov chains
and introduce the concept of multistable states. Our presentation unifies many related concepts concepts and shows which
structural (deterministic, graph theoretic, semiflow) properties and
which probabilistic properties are really needed to analyze Markov
chains. Recall that a hybrid system is a dynamical system exhibiting the interaction between discrete and continuous phenomena. In this work we lay the foundation for a dynamical systems approach to hybrid systems with Markov chain type perturbations.


\section{Orbit decomposition of finite directed graphs}\label{secgraph}

We start by presenting basic definitions and notations used along this work. Our first goal is to produce two simple graph decompositions motivated by concepts form dynamical systems. In contrast to the standard convention, throughout this work, the {\it vertex communication is not in necessarily an equivalence relation}. This simple observation will lead to interesting phenomena. We define communicating sets and classes, and a partial order between communicating classes is presented. We introduce a necessary condition to meaningfully study asymptotic behaviour via the concept of orbit. Some of the results presented in this section are somehow elementary, we prefer to include such results for the sake of exposition.  



\begin{remark}\label{nomult}Throughout this note whenever we refer to a {\it graph} $G=(V,E)$ we mean a {\it finite directed graph} when $V$ is the set of vertices in $G$ and $E$ is the set of edges in $G$ such that {\it no multiple edges between any two elements in $V$ are allowed}.
\end{remark}

\subsection{Orbits and communicating classes}  \label{orbits&cc} The communication structure in graphs is a central concept in this work. In this section we introduce the concepts of communicating sets and communicating classes based on the idea of orbits. 

Let $G$ be a graph. An edge from the vertex $i$ to the vertex $j$ is denoted $(i,j)\in E$. A path in $G$ correspond to a sequence of vertices agree with the incidence and direction in $G$ and is denoted by $\langle i_0i_1\ldots i_n \rangle$. Sometimes we write $i\in \gamma$ to specify that the vertex $i$ belongs to the path $\gamma$. We define the set:
\begin{equation}
\Gamma^{n}=\{\gamma:\ell(\gamma)=n,\,\ n\in\mathbb{N}\}\label{defpathset}%
\end{equation}
as the set of all paths $\gamma$ of {\it length} $\ell(\gamma)=n$, with $\Gamma^{0}=V$. We can specify vertices in a path $\gamma$ in terms of the projection maps
$\pi_{p}$ for
$0\leq p\leq n$:%
\[
\pi_{p}:\Gamma^{n}\rightarrow V\text{, }\ \pi_{p}(\gamma)=i_{p}%
\]
where $i_{p}$\ is the $p^{th}$ vertex in $\gamma$. In other words,
\[
\gamma=\langle
\pi_{0}(\gamma)\ldots \pi_{p}(\gamma)\ldots \pi_{n}(\gamma
) \rangle\text{.}%
\]

A subpath $\gamma^{\prime}$ of $\gamma$ is a subsequence
of $\gamma$ of consecutive edges (or vertices) belonging to
$\gamma$. In particular, any edge of a path is a subpath of length
one. Composition of paths will play a role in many
of our proofs.

\begin{definition}
For two paths $\gamma_{1}$ and $\gamma_{2}$ with
$\ell(\gamma_{1})=m$ and $\ell(\gamma_{2})=n$ such that
$\gamma_{1}=\langle i\ldots j \rangle$ and $\gamma_{2}=\langle j \ldots k
\rangle$ we define the concatenation of the paths
as%
\[
\left\langle \gamma_{1},\gamma_{2}\right\rangle =\left\langle
i \ldots j\right\rangle \ast\left\langle j \ldots k\right\rangle =\left\langle
i \ldots k\right\rangle
\]
with $\ell(\langle i \ldots k \rangle)=m+n$ and
$\pi_{m}(\langle  i \ldots k \rangle)=j$.
\end{definition}


\begin{definition}\label{defcs}A vertex $i\in V$ has access to a vertex $j\in V$ if there exists a path of
length $\geq1$ from $i$ to $j$. We say that the vertices $i$ and $j$
{\it communicate}, written as $i\sim j$, if they have mutual access. A subset $U$ of
$V$ is a {\it communicating set} if any two vertices of $U$ communicate.
\end{definition}

%

\begin{remark} Note that the relation $\sim$ is reflexive iff for
all $i\in V$ there exists a path $\gamma_{ii}=\left\langle i \ldots i\right\rangle
$, a property that does not always hold. In general we have that: The vertex communication relation $\sim$ is symmetric and
transitive but, in general, it lacks the reflexivity property.
\end{remark}


We define a
smaller set on which reflexivity holds. Denote the union of all
communicating sets by:
\[
V_{c}=\{\,\,i\in V:\,\,i\sim j\text{ \ for some \ }j\in V\,\,\}\text{.}%
\]

For $i\in V_{c}$ we define $[i]=\{j\in V$, $j\sim i\}$. Then
$V_{c}/\hspace{-.1cm}\sim=\{[i]$, $i\in V_{c}\}$ is a partition of $V_{c}$, that is,
$[i]\cap\lbrack j]=\varnothing$ for $j\notin\lbrack i]$, and $\cup\lbrack
i]=V_{c} $.

\begin{definition}
\label{defcc}Let $G$ be a graph with communication relation $\sim$. Each set
$[i]$ for $i\in V_{c}$ is called a {\it communicating class} of $G$. We denote the
set $V_{c}/\hspace{-.15cm}\sim$ of all communicating classes by $\mathcal{C}$.
\end{definition}

Note that by definition, communicating classes are communicating
sets. They are characterised by their maximality.

\begin{proposition} Communicating classes are maximal communicating sets with respect to the set
inclusion. Vice versa, maximal communicating sets are communicating classes.
\end{proposition}

\begin{proof}
Assume that $[j]\in\mathcal{C}$ is not maximal, then there exist $i\in
[j]$, and $k\notin[j]$ with $i\sim k$ i.e.,
$[j]$ is not maximal. Since $i\in[j]$ there exist paths
$\gamma_{1}=\langle i \ldots j \rangle$ and $\gamma_{2}=\langle j \ldots i \rangle$.
Moreover, since $i$ and $k$ communicate we have paths $\gamma_{3}%
=\langle i \ldots k \rangle$ and $\gamma_{4}=\langle k \ldots i \rangle$. The
concatenations $\langle \gamma_{4},\gamma_{1}\rangle$ and $\langle
 \gamma_{2},\gamma_{3} \rangle$ imply $k\sim j$ and therefore $k\in
[j]$, which leads to a contradiction, proving the first claim of
the proposition. The second part follows by definition of communicating classes.
\end{proof}

\begin{definition}
\label{deforbit} The positive and negative orbit of a vertex $i\in V$
are
defined as:%
\[
\mathcal{O}^{+}(i)=\{j\in\ V:\,\,\,\
\exists\,n\geq1,\,\,\,\exists\,\gamma
\in\Gamma^{n}\text{ \ such that \ }\pi_{0}(\gamma)=i,\,\,\,\pi_{n}%
(\gamma)=j\}\text{,}%
\]%
\[
\mathcal{O}^{-}(i)=\{j\in\ V:\,\,\,\
\exists\,n\geq1,\,\,\,\exists\,\gamma
\in\Gamma^{n}\,\,\text{such that \ }\pi_{0}(\gamma)=j,\,\,\,\pi_{n}%
(\gamma)=i\}\text{,}%
\]
where $\pi_{0}(\gamma)$ and $\pi_{n}(\gamma)$ represent the initial
and the final vertices in $\gamma$.
\end{definition}

Our first result shows that communicating classes can be
characterised using orbits of vertices.

\begin{theorem}
\label{cc=orbit}Every communicating class $C\in\mathcal{C}$ is of the form
\[
C=\mathcal{O}^{+}(i)\cap\mathcal{O}^{-}(i)
\]
for some $i\in V$. Vice versa, if $C:=\mathcal{O}^{+}(i)\cap\mathcal{O}%
^{-}(i)\neq\varnothing$ for some $i\in V$, then $C$ is a communicating class.
\end{theorem}

\begin{proof} Let $C$ be a communicating class with $i\in C$. Then since $i\sim
i$, we have that $C$ contains a path $\gamma=\langle i \ldots i \rangle$ and
hence it follows that $i\in\mathcal{O}^{+}(i)\cap\mathcal{O}^{-}(i)$, that is,
 $C\subset\mathcal{O}^{+}(i)\cap\mathcal{O}^{-}(i)$. Now consider
$j\in\mathcal{O}^{+}(i)\cap\mathcal{O}^{-}(i)$ for some $i\in V$. Then there
exist a path $\gamma$ and $n\geq1$ such that $\pi_{0}(\gamma)=i$ and $\pi
_{n}(\gamma)=j$, as well as a path $\gamma^{\prime}$ such that $\pi_{0}%
(\gamma^{\prime})=j$ and $\pi_{m}(\gamma^{\prime})=i$, for some $m\geq1$. This
immediately implies $i\sim j$ for every element $j\in\mathcal{O}^{+}%
(i)\cap\mathcal{O}^{-}(i)$, and therefore $\mathcal{O}^{+}(i)\cap
\mathcal{O}^{-}(i)\subset\lbrack i]$. Conversely, assume that $\mathcal{O}^{+}(i)\cap\mathcal{O}^{-}(i)\neq
\varnothing$ for some $i\in V$. We have to show that $\mathcal{O}^{+}%
(i)\cap\mathcal{O}^{-}(i)$ is a communicating class, i.e. $\mathcal{O}%
^{+}(i)\cap\mathcal{O}^{-}(i)=[i]$. Take $j\in\mathcal{O}^{+}(i)\cap
\mathcal{O}^{-}(i)$, then we argue as before that $j\sim i$ and hence
$j\in\lbrack i]$. On the other hand, if $j\notin\mathcal{O}^{+}(i)\cap
\mathcal{O}^{-}(i)$, then $j\notin\mathcal{O}^{+}(i)$ or $j\notin
\mathcal{O}^{-}(i)$. In the first case there is no path from $i$ to $j$, in
the second case there is no path from $j$ to $i$. Any of these two statements
implies that $i\nsim j$, which completes the proof.
\end{proof}

\begin{definition}
\label{deftransitory} A transitory vertex is a vertex that
does not belong to a communicating class.
\end{definition}

Note that by Theorem \ref{cc=orbit} transitory vertices are exactly those
vertices $i\in V$ for which $\mathcal{O}^{+}(i)\cap\mathcal{O}^{-}%
(i)=\varnothing$. This also means that $\mathcal{O}^{+}(i)\cap\mathcal{O}%
^{-}(i)\neq\varnothing$ iff $i\in V_{c}$ i.e., exactly these vertices {\it anchor}
communicating classes. In addition, the statements in Theorem \ref{cc=orbit} take on this simple form
because we have defined orbits in Definition \ref{deforbit} as starting with
paths of length 1, not 0. If we include paths of length 0 in an orbit, then it
always holds that $i\in\mathcal{O}^{+}(i)\cap\mathcal{O}^{-}(i)$. This trivial
situation then needs to be excluded in Theorem \ref{cc=orbit}. Similarly, we
have defined communicating classes in Definition \ref{defcc} using mutual
access (a vertex $i\in V$ satisfies $i\in V_{c}$ if there exists a path
of length $\geq1$ from $i$ to $i$). This avoids the triviality that each vertex
communicates with itself. Note that for systems on continuous state spaces one
needs separate non-triviality conditions, such as the existence of an infinite
path within a  {\it communicating class} and a condition on the richness of the
orbits, see the discussion in \cite{CK00}, Chapter 3 for control systems.


\subsection{Communicating sets in $L-$graphs} We analyse communicating structure in graphs that admit limit behaviour. 

\begin{definition} \label{inout-deg}
In a directed graph $G$ the out-degree of a vertex $i\in
V$ is defined as the number of edges {\it going out of
the vertex $i$}
\[
O(i)=\#\,\{(i,j):\,\,\,(i,j)\in E\,\,\,\text{for some
}\,j\in V\}.
\]
The in-degree of a vertex $i$ corresponds to the number
of edges {\it coming into $i$}
\[
I(i)=\#\,\{(j,i):\,\,\,(j,i)\in E\,\,\,\text{for some
}\,j\in V\}.
\]
\end{definition}

\begin{definition}
\label{deflgraph} A graph is called an $L-$graph if every vertex has positive out-degree.
\end{definition}

From now on we only concentrate on a graph $G=(V,E)$ such that $O(i)\geq 1$ for all $i\in V$. $L-$graphs are needed to ensure the existence of various objects, such as communicating classes. Definition \ref{deflgraph} states a non-degeneracy condition for orbits in
a graph. We ensure the existence of communicating classes (with certain
additional properties). 


\begin{lemma}
\label{lpath}Each $L-$graph has paths of arbitrary length.
\end{lemma}

\begin{proof}
Let $G$ be an $L-$graph and $n\geq1$.\ Pick $i_{0}\in V$, then by definition
${O}(i_{0})\,\geq{1}$. This ensures the existence of $i_{1}\in V$ such
that $(i_{0},i_{1})$ is an incidence in $G$. Using the same argument we see
that there is a vertex $i_{2}\in V$ and an incidence $(i_{1},i_{2})$.
Continuing with this process up to step $n$ we infer the existence of a path%
\[
\gamma=\langle i_{0},(i_{0},i_{1}),i_{1},(i_{1},i_{2}), \ldots ,i_{n-2}%
,(i_{n-1},i_{n}),i_{n} \rangle
\]
or equivalently,%
\[
\gamma=\langle i_{0} \ldots i_{n} \rangle
\]
with $\ell(\gamma)=n$.
\end{proof}

\begin{definition}
A path $\gamma$ of length $\ell(\gamma)=n$, with $n\geq1$, is said to be a loop if
there exists a vertex $i\in\gamma$ such that $\pi_{0}(\gamma)=\pi_{n}%
(\gamma)=i$.
\end{definition}

\begin{lemma}
\label{pathloop}In a graph $G$ with $d$ vertices any path $\gamma$ of length
$\ell(\gamma)=n$, with $n\geq d$ contains a loop.
\end{lemma}

\begin{proof}
We consider a path $\gamma$ in $G$ such that $\gamma=\langle i_{0}i_{1}\ldots i_{d} \rangle$
with $\ell(\gamma)=d$. Assume that the subpath $\langle i_{0} \ldots i_{d-1}%
 \rangle$ contains no loop. Then all the vertices of
$\langle i_{0} \ldots i_{d-1} \rangle$ are distinct and hence $\{ i_{0}%
,\ldots,i_{d-1} \}=V$. Now $i_{d}\in V$\ implies that there exists $\alpha
\in\{0,\ldots,d-1\}$\ with $i_{d}=i_{\alpha}.$\ Hence the subpath $\langle
 i_{\alpha} \ldots i_{d} \rangle$ is a loop contained in $\gamma$.
\end{proof}

\noindent The next three lemmata explore the relationship between loops and
communicating classes, leading to the existence of communicating classes in
$L$-graphs.

\begin{lemma}
\label{loopcc}If $G$ has a loop, then there exists
a communicating class $C$ in $G$ such that the vertices in the loop are
contained in $C$.
\end{lemma}

\begin{proof}
Consider a loop $\lambda=\langle i_{0}i_{1} \ldots i_{n}i_{0} \rangle$. Since
the elements in $\lambda$ have mutual access each other, we have $i_{k}%
\in[i_{0}]$ for $0\leq k\leq{n}$. Therefore each $i_{k} $
belongs to the same communicating class $C=[{i_{0}}]$.
\end{proof}

\begin{lemma}
\label{pathout}Let $G$ be a graph and $\gamma=\langle i_{0}\ldots i_{n} \rangle$
a path in $G$. If there is a communicating class $C$ with $i_{0}\in
C$ in $G$, and if there is $\alpha\in\{1,\ldots,n\}$ with $i_{\alpha}\notin C$, then
$i_{\beta}\notin C$ for all $\beta\in\{\alpha,\ldots,n\}$.
\end{lemma}

\begin{proof}
Using the notation of the statement of the lemma, assume, to the contrary,
that there exists $\beta\geq\alpha$ with $i_{\beta}\in C$. Then there are
paths $\gamma_{1}=\langle i_{0} \ldots i_{\alpha} \rangle$ and $\gamma
_{2}=\langle i_{\alpha} \ldots i_{\beta} \ldots i_{0} \rangle$, showing that
$i_{0}\sim i_{\alpha}$ and hence $i_{\alpha}\in C$, which is a contradiction.
\end{proof}

\begin{lemma}
\label{ccloop}If a vertex $i$ belongs to a communicating class
$C$ in $G$, then there exists a loop $\lambda$ in $C$ such that $i\in
\lambda$.
\end{lemma}

\begin{proof}
Let $G$ be a graph and $C$ a communicating class in $G$. Then there is
pairwise communication between the elements in $C$, i.e., given $i,j\in C$
there exists a path $\gamma_{ij}=\langle i \ldots j \rangle$. In particular, for
$i=j$ we have the path $\gamma_{ii}=\langle i \ldots i \rangle$ with $\pi
_{0}(\gamma)=\pi_{n}(\gamma)=i$ for some $n\geq1$. Hence $\gamma$ is indeed a
loop containing $i$. By Lemma \ref{pathout} all components of this loop are in
$C$.
\end{proof}

\begin{proposition}
\label{lhascc}An $L-$graph has at least one communicating class.
\end{proposition}

\begin{proof}
Consider an $L-$graph $G$ with $d$ vertices. By Lemma \ref{lpath} there exists a
path $\gamma$ such that $\ell(\gamma)=n$ with $n\geq d$. By Lemma
\ref{pathloop} the path $\gamma$ contains a loop $\lambda$, and by Lemma
\ref{loopcc} there exist a communicating class containing the vertices of
$\lambda$.
\end{proof}



\begin{remark}
\label{ccinorbit}The proof of Proposition \ref{lhascc} actually shows the
stronger statement: Let $G$ be an $L-$graph and $i\in V$. Then there exists at
least one communicating class $C$ with $C\subset\mathcal{O}^{+}(i)$. Note
that, in general, $i\in\mathcal{O}^{+}(i)$ may not hold.
\end{remark}


\noindent As a final idea of this section we explore an order on the set of
communicating classes, which will lead to a characterisation of so-called
forward invariant classes.

\begin{definition}
\label{deforder}Let $G$ be a graph with a family $\mathcal{C}=\{C_{1}%
,\ldots,C_{k}\}$ of communicating classes. We define a relation on $\mathcal{C}$
by%
\[%
\begin{array}
[c]{ccc}%
C\mathcal{_{\mu}\preceq}C\mathcal{_{\nu}} & \text{if} & \text{there exists a
path }\gamma\in\Gamma^{n}\text{ with}\\
&  & \pi_{0}(\gamma)\in C\mathcal{_{\mu}}\text{ and }\pi_{n}(\gamma)\in
C\mathcal{_{\nu}}\text{.}%
\end{array}
\]

\end{definition}

\begin{lemma}
\label{ccorder}Let $G$ be a graph with a family $\mathcal{C}=\{C_{1}%
,\ldots,C_{k}\}$ of communicating classes. The relation $\preceq$ defines a
(partial) order on $\mathcal{C}$.
\end{lemma}

\begin{proof}  {\it Reflexivity}: Let $C\mathcal{_{\mu}}$ be a communicating class in
$G$. By Lemma \ref{ccloop}, for any $i\in C\mathcal{_{\mu}}$ there exists a
loop $\lambda$ in $C\mathcal{_{\mu}}$ such that $i\in\lambda$. Since
$C\mathcal{_{\mu}}$ is a communicating class, the loop $\lambda$ gives us a
path form $C\mathcal{_{\mu}}$ to itself, and therefore $C\mathcal{_{\mu
}\preceq}C\mathcal{_{\mu}}$.

   {\it Antisymmetry}: Assume that $C_{\mu}\preceq C_{\nu}$ and $C_{\nu
}\preceq C_{\mu}$. From the first relation we get a path $\gamma_{\mu\nu}%
\in\Gamma^{p}$ for some $p\geq1$ such that
\[
\pi_{0}(\gamma_{\mu\nu})\in C_{\mu}\,\,\,\,\,\text{and}\,\,\,\,\,\pi
_{p}(\gamma_{\mu\nu})\in C_{\nu}\text{.}%
\]
By the second relation there exists a path $\gamma_{\nu\mu}\in\Gamma^{q} $ for
some $q\geq1$ with
\[
\pi_{0}(\gamma_{\nu\mu})\in\mathcal{C}_{\nu}\,\,\,\,\,\,\text{and}%
\,\,\,\,\,\pi_{q}(\gamma_{\nu\mu})\in\mathcal{C}_{\mu}.
\]
Hence the concatenation $\left\langle \gamma_{\mu\nu},\gamma_{\nu\mu
}\right\rangle $\ is a path in $C_{\mu}$\ of length $p+q$.\ Since
communicating classes are maximal, and any two vertices in a communicating
class have mutual access, it holds that $\mathcal{C_{\mu}\ =\,\,\,C_{\nu}}$.

  {\it Transitivity}: Let $\mathcal{C}_{\mu}$, $\mathcal{C}_{\nu}$,
$\mathcal{C}_{\xi}$ in $\mathcal{C}$ and suppose that $C_{\mu}\preceq C_{\nu}$
and $C_{\nu}\preceq C_{\xi}$ hold. Since $C_{\mu}\preceq C_{\nu}$, there
exists a path $\gamma_{1}$ such that $\pi_{0}(\gamma_{1})\in C_{\mu}$ and
$\pi_{n_{1}}(\gamma_{1})\in C\mathcal{_{\nu}}$ for some $n_{1}\in\mathbb{N}$.
Moreover, since $C_{\nu}\preceq C_{\xi}$, there exists a path $\gamma_{3}$
such that $\pi_{0}(\gamma_{3})\in C_{\nu}$ and $\pi_{n_{3}}(\gamma_{3})\in
C_{\xi}$ for some $n_{3}\in\mathbb{N}$. Since $C\mathcal{_{\nu}}$ is a
communicating class, there exists a path $\gamma_{2}$ in $C\mathcal{_{\nu}}$
such that $\pi_{0}(\gamma_{2})=\pi_{n_{1}}(\gamma_{1})$ and $\pi_{n_{2}%
}(\gamma_{2})=\pi_{0}(\gamma_{3})$ for some $n_{2}\in\mathbb{N}$. The path
$\gamma=\langle\gamma_{1},\gamma_{2},\gamma_{3}\rangle$ is such that $\pi
_{0}(\gamma)\in C_{\mu}$ and $\pi_{m}(\gamma)\in C_{\xi}$ for $m=n_{1}%
+n_{2}+n_{3}$. Therefore, $C_{\mu}\preceq C_{\xi}$.

\end{proof}



\begin{definition}
\label{definvarsetgraph}A set of vertices $U$ of $G$ is called forward
invariant if,
\[
\mathcal{O}^{+}(U)\subset U\text{.}%
\]
Similarly, $U$ is called backward invariant if,
\[
\mathcal{O}^{-}(U)\subset U%
\]
and invariant if,
\[
\mathcal{O}^{+}(U)\cup\mathcal{O}^{-}(U)\subset U\text{.}%
\]

\end{definition}

\begin{remark}
\label{ordermax}Note that, by definition, a forward invariant communicating
class is maximal with respect to the order $\preceq$ introduced in Definition
\ref{deforder}.
\end{remark}

\begin{proposition}
\label{lficc}An $L-$graph $G$ contains a forward invariant communicating class.
\end{proposition}

\begin{proof}
Consider an $L-$graph $G$ with set of communicating classes $\mathcal{C}%
=\{C_{1}, \ldots ,C_{k}\}$. Since an order relation on a finite set has maximal element, denote by $C_{\mu}$ a maximal element in $(\mathcal{C},\preceq)$. We show that $C_{\mu}$\ is forward invariant:
Assume to the contrary that $C_{\mu}$ is not forward invariant. Since then
$\mathcal{O^{+}}(\,C_{\mu}\mathcal{\,)}$ is not contained in $C_{\mu}$,there
exist $i,j_{0}\in V$ with $i\in C_{\mu}$ and $j_{0}\notin C_{\mu}$ such that
$(i,j_{0})\in E$. Since $O(j_{0})\geq1$, there exists $j_{1}\in V$
such that $(j_{0},j_{1})\in E$. Note that by Lemma \ref{pathout} $j_{1}$
cannot belong to$\mathcal{\ }C_{\mu}$. By Remark \ref{ccinorbit}\ there exists
a communicating class $C\subset\mathcal{O}^{+}(j_{0})$ and $C\cap C_{\mu
}=\varnothing$ by Lemma \ref{pathout}. It follows that, $C_{\mu}\preceq C$,
which contradicts maximality of $C_{\mu}$.
\end{proof}

\begin{remark}
\label{ficcinorbit}Note that in the proof of Proposition \ref{lficc} we
actually showed the stronger statement: Let $G$ be an $L-$graph and $i\in V$.
Then $\mathcal{O}^{+}(i)$ contains a forward invariant communicating class.
\end{remark}

\begin{remark}
\label{graphsum}Summarizing Remark \ref{ordermax} and Proposition \ref{lficc}
we see that for an $L-$graph the maximal elements of $(\mathcal{C},\preceq)$
are exactly the forward invariant communicating classes. It also follows
directly from Definition \ref{deforder} that backward invariant communicating
classes are minimal in $(\mathcal{C},\preceq)$. Note, however, that minimal
communicating classes in $(\mathcal{C},\preceq)$ need not be backward
invariant. For this fact to hold we would need a backward nondegeneracy
condition similar to the $L-$graph property, e.g., using the in-degree of
vertices. For an analogue of this issue in the theory of control systems in
discrete time compare \cite{AS91}.
\end{remark}

\subsection{Quotient graphs}\label{quograph}
There are at least two quotient structures associated with the idea of
communicating sets in a graph. The first idea is to simply take the order
graph given in Definition \ref{deforder}. Equivalently,  this graph is obtained as the quotient
$V_{c}/ \hspace{-.1cm}\sim$. The graph obtained from Definition \ref{deforder} does not necessarily cover all the vertices of a
given graph $G$, and its edges may not be edges of $G$. For a given directed graph $G=(V,E)$\ we define $V_{Q}%
:=\mathcal{C}\cup V\backslash V_{c}=\{C$, $C$ is a communicating
class$\}\cup\{i\in V$, $i$ is a transitory vertex$\}$ as a set of vertices.
The set of edges is constructed as follows: For $A,B\in V_{Q}$ we set
$(A,B)\in E_{Q}$ if there exist $i\in A$ and $j\in B$ with $(i,j)\in E$. (Note
the abuse of notation: if $A\in V_{Q}$ is a (transitory) vertex of $G$ then
$``i\in A"$ is to be interpreted as $`` i=A"$.) The graph $G_{Q}=(V_{Q},E_{Q})$
is called the extended quotient graph of $G$. It is easily seen that the
extended quotient graph of $G_{Q}$\ is $G_{Q}$\ itself. All vertices in
$V_{Q}$ have specific interpretations in the context of Markov chains, see
Section \ref{dictionary}.




\section{The semiflow of a finite directed graph \label{secsemiflow}}

Our second approach to study decompositions of graphs is based on an idea from
the theory of dynamical systems. A Morse
decomposition describes the global behaviour of a dynamical system, i.e. the
limit sets of a system and the flow between these sets. The following definition and basic ingredients are standard. These can be found for example in \cite{AH06, CK00}.

\begin{definition}
\label{MorseDecomp}A \textit{Morse decomposition} of a flow on a compact
metric space $X$ is a finite collection $\left\{  \mathcal{M}_{i}%
,\;i=1,...,n\right\}  $ of nonvoid, pairwise disjoint, and compact isolated
invariant sets such that:\smallskip
\begin{itemize}
 \item For all $x\in X$ one has $\omega(x),\,\omega^{\ast}(x)\subset%
{\displaystyle\bigcup_{i=1}^{n}}
\mathcal{M}_{i}$.


\item Suppose there are $\mathcal{M}_{j_{0}},\mathcal{M}_{j_{1}%
},...,\mathcal{M}_{j_{l}\text{ }}$ and $x_{1},...,x_{l}\in X\setminus%
{\displaystyle\bigcup_{i=1}^{n}}
\mathcal{M}_{i}$ with $\omega^{\ast}(x_{i})\subset\mathcal{M}_{j_{i-1}}$ and
$\omega(x_{i})\subset\mathcal{M}_{j_{i}}$ for $i=1,...,l$; then $\mathcal{M}%
_{j_{0}}\neq\mathcal{M}_{j_{l}}$.
\end{itemize}
The elements of a Morse decomposition are called \textit{Morse sets}.
\end{definition}

A Morse decomposition results in an order among the components, the Morse sets, of the
decomposition. It can be constructed from attractors and repellers, and the
behaviour of the system on the Morse sets is characterised by (chain) recurrence. Here we then construct an
analogue for (discrete) systems defined by $L-$graphs. Unfortunately,
this analogy is not complete since systems defined by these directed graphs only
lead to semiflows, i.e. systems for which the time set is $\mathbb{N}$, and
not all of $\mathbb{Z}$.

\subsection{Semiflows associated with graphs }\label{graphsemiflow}

When  adapting the idea of Morse decompositions and attractors-repellers
to systems induced by directed graphs, one faces two main
challenges: The first concerns the topology on discrete spaces that has some
interesting consequences for limit sets, isolated invariant\ sets, etc. The
second challenge stems from the fact that the out-degree of vertices can be
$>1$, resulting in set-valued systems. This is the reason why graphs define
semiflows (on $\mathbb{N}$ instead of $\mathbb{Z}$).

 Let $G=(V,E)$ be a finite directed graph. We denote by $\mathcal{P}(V)$ the power set of
the vertex set. We consider the topology derived by the discrete metric on $\mathcal{P}(V)$.

The directed graph $G$ gives rise to two semiflows, one with
the positive integers $\mathbb{N}$\ as time set, and one with the negative
integers $\mathbb{N}^{-}:=\{-n$, $n\in\mathbb{N}\}$:%
\[
\Phi_{G}:\mathbb{N}\times\mathcal{P}(V)\rightarrow\mathcal{P}(V)\text{, }%
\]%
\begin{equation}
\Phi_{G}(n,A)=\left\{
\begin{array}
[c]{cc}%
j\in V: & \exists\text{ }i\in A\text{ and }\gamma\in\Gamma^{n}\text{ such
that}\\
& \pi_{0}(\gamma)=i\text{ and }\pi_{n}(\gamma)=j
\end{array}
\right\}  \text{.}\label{defflow+}%
\end{equation}
Similarly, we define
\[
\Phi_{G}^{-}:\mathbb{N}^{-}\times\,\,\mathcal{P}(V)\rightarrow\mathcal{P}(V)\text{, }%
\]%
\begin{equation}
\Phi_{G}^{-}(n,A)=\left\{
\begin{array}
[c]{cc}%
j\in V: & \exists\text{ }i\in A\text{ and }\gamma\in\Gamma^{-n}\text{ such
that}\\
& \pi_{0}(\gamma)=j\text{ and }\pi_{-n}(\gamma)=i
\end{array}
\right\}  \text{.}\label{defflow-}%
\end{equation}

 We note that the definition of the semiflows $\Phi_{G}(n,A)$ and
$\Phi_{G}^{-}(n,A)$ only requires the basic ingredients of a graph: the sets
of vertices and of edges. 


 The next proposition collects some properties of the maps defined in
(\ref{defflow+}) and (\ref{defflow-}).

\begin{proposition}
\label{semiflow}Consider a graph $G$ and the associated map $\Phi_{G}$
defined in (\ref{defflow+}). This map is a semiflow, i.e. it has the properties:

\begin{itemize}
\item $\Phi_{G}$ is continuous,

\item $\Phi_{G}(0,A)=A$ for all $A\in\mathcal{P}(V)$,

\item $\Phi_{G}(n+m,A)=\Phi_{G}(n,\Phi_{G}(m,A))$ for all $A\in\mathcal{P}(V)$,
$m,n\in\mathbb{N}$.
\end{itemize}

\noindent The same properties hold for the negative semiflow $\Phi_{G}^{-}$.
\end{proposition}

\begin{proof}
Note that in the discrete topology every function is continuous. The second item holds by definition of $\Gamma^{0}$ in (\ref{defpathset}): The
vertices reached from $A$ under the flow at time zero, are the elements in $V$
that belong to $A$. To show the third item we first assume that the three sets
$\Phi_{G}(m,A)$,\ $\Phi_{G}(n+m,A)$, and $\Phi_{G}(n,\Phi_{G}(m,A))$\ are
nonempty. Consider $j\in\Phi_{G}(n+m,A)$: There exist $i\in A$ and a path
$\alpha\in\Gamma^{n+m}$ such that $\pi_{0}(\alpha)=i$ and $\pi_{n+m}%
(\alpha)=j$. We split $\alpha$ as the concatenation of two paths $\beta$ and
$\gamma$ with $\ell(\beta)=m$ and $\ell(\gamma)=n$, having, $\pi_{0}(\beta
)=i$, $\pi_{m}(\beta)=k$ and $\pi_{0}(\gamma)=k$, $\pi_{n}(\gamma)=j$ for some
$k\in V$. Observe that by definition we have $k\in\Phi_{G}(m,A)$ and hence
$j\in\Phi_{G}(n,\Phi_{G}(m,A))$. On the other hand, if $j\in\Phi_{G}%
(n,\Phi_{G}(m,A))$ then there exist $k\in\Phi_{G}(m,A)$ and $\beta\in
\Gamma^{n}$ such that $\pi_{0}(\beta)=k$ and $\pi_{n}(\gamma)=j\in\Phi
_{G}(n,\Phi_{G}(m,A))$. Since $k\in\Phi_{G}(m,A)$, there are $i\in A$ and
$\alpha\in\Gamma^{m}$ such that $\pi_{0}(\alpha)=i$ and $\pi_{m}(\alpha)=k.$
The concatenation $\gamma=\left\langle \alpha,\beta\right\rangle
\ $satisfies$\ \pi_{0}(\gamma)=i$ and $\pi_{m+n}(\gamma)=j$, hence $j\in
\Phi_{G}(n+m,A)$.

 To finish the proof for $\Phi_{G}$, we consider the case that (at
least) one of the three sets $\Phi_{G}(m,A)$,\ $\Phi_{G}(n,\Phi_{G}(m,A))$,
and\ $\Phi_{G}(n+m,A)$\ is empty: Note first of all that by the definition of
$\Phi_{G}$ in (\ref{defflow+}) we have $\Phi_{G}(n,\varnothing)=\varnothing$
for all $n\in\mathbb{N}$. (i) If $\Phi_{G}(m,A)=\varnothing$ then we have
$\Phi_{G}(n,\Phi_{G}(m,A))=\varnothing$ by the preceding argument. Now if
$\Phi_{G}(n+m,A)\neq\varnothing$, then we can construct, as in the previous
paragraph, a point $k\in\Phi_{G}(m,A)$ by splitting a path $\alpha\in
\Gamma^{n+m}$ with $\pi_{0}(\alpha)\in A$ and $\pi_{n+m}(\alpha)\in\Phi
_{G}(n+m,A)$. This contradicts $\Phi_{G}(m,A)=\varnothing$ and therefore
$\Phi_{G}(n+m,A)=\varnothing$. (ii) Assume next that $\Phi_{G}(n,\Phi
_{G}(m,A))=\varnothing$. Using the same reasoning from the previous paragraph,
we see that if $j\in\Phi_{G}(n+m,A)$ then $j\in\Phi_{G}(n,\Phi_{G}(m,A))$.
Hence it holds that $\Phi_{G}(n+m,A)=\varnothing$. (iii) If $\Phi
_{G}(n+m,A)=\varnothing$ then there exists no path $\gamma\in\Gamma^{n+m}$
with $\pi_{0}(\gamma)\in A$. But by the reasoning in the previous paragraph,
if $j\in\Phi_{G}(n,\Phi_{G}(m,A))$ then there exist $i\in A$ and $\gamma
\in\Gamma^{m+n}$ such that $\pi_{0}(\gamma)=i $ and $\pi_{m+n}(\gamma)=j$,
which cannot be true, and hence we see that $\Phi_{G}(n,\Phi_{G}%
(m,A))=\varnothing$.

 The proof for $\Phi_{G}^{-}$ follows the same lines.
\end{proof}

\begin{remark}
\label{lflows+}If $G$ is an $L-$graph, then $\Phi_{G}(n,A)\neq
\varnothing$ for $A\neq\varnothing$ and $n\in\mathbb{N}$. This observation may
not hold for the negative semiflow $\Phi_{G}^{-}$ without additional assumptions.
\end{remark}

 One might wonder if the semiflows $\Phi_{G}$ defined in
(\ref{defflow+}) and $\Phi_{G}^{-}$\ from (\ref{defflow-}) can be combined to
a flow on $\mathbb{Z}$. It is not hard to see that in general this is
not possible even if the graph $G$ has additional properties (such as being
an $L-$graph) or if one restricts oneself to graphs that are one communicating class.

\begin{remark}
The semiflow $\Phi_{G}^{-}$ as defined in (\ref{defflow-}) can be interpreted
as the positive semiflow of the graph $G^{T}=(V,E^{T})$, where $(i,j)\in
E^{T}$ iff $(j,i)\in E$. $\Phi_{G}^{-}$ is sometimes called the time-reverse
semiflow of $\Phi_{G}$. Under the corresponding assumptions, all statements
for a positive semiflow also hold for its time-reverse counterpart.
\end{remark}

\subsection{Morse decompositions of semiflows}\label{semiMorsedecomp}
We next turn to some concepts from the theory of dynamical systems and study
their analogues for the semiflows defined in (\ref{defflow+}) and
(\ref{defflow-}).
Note the state space $\mathcal{P}(V)$ of the
semiflow $\Phi_{G}$ is finite with the discrete topology. It suffices to
introduce the concepts for points $A\in\mathcal{P}(V)$. To avoid trivial
situations where $\Phi_{G}(n,A)=\varnothing$ for some $n\in\mathbb{N}$ and
$A\in\mathcal{P}(V)$, $A\neq\varnothing$, we assume that all graphs are
$L-$graphs, compare Remarks \ref{lflows+} above and \ref{l=limitnonempty}%
\ below.

Next we adapt the necessary ingredients for a meaningful Morse decomposition of a semiflow in an $L-$graph.
 
 \textbf{Invariance:} A point $A\in\mathcal{P}(V)$ is said to be (forward)
invariant if $\Phi_{G}(n,A)\subset A$ for all $n\in\mathbb{N}$. Note that
$A\in\mathcal{P}(V)$ is invariant under $\Phi_{G}$ iff $A$ is a forward invariant
set of the underlying graph $G=(V,E)$, compare Definition
\ref{definvarsetgraph}. Hence invariance under $\Phi_{G}$ is a fairly strong
requirement of a set $A\in\mathcal{P}(V)$. As we will see, a meaningful Morse
decomposition of the semiflow $\Phi_{G}$ only requires a weak form of invariance.

\begin{definition}
A point $A\in\mathcal{P}(V)$ is said to be weakly invariant if for all
$n\in\mathbb{N}$ we have $\Phi_{G}(n,A)\cap A\neq\varnothing$.
\end{definition}

 \textbf{Isolated invariance:} For a (forward) invariant set $A\in\mathcal{P}(V) $
one could define {\it forward isolated invariant}.
\hspace{.1cm}But because of the discrete topology, we could choose
$N(A)=A $, and any forward invariant set then satisfies this property. As we
will see, because of the discrete topology employed, a meaningful Morse
decomposition of the semiflow $\Phi_{G}$ does not require the property of
isolated invariance.

\textbf{Limit sets:} To adapt the concept of a limit set from 
 the standard definition to the semiflow $\Phi_{G}$, note that a sequence converges
in the discrete topology iff it is eventually constant. Hence limit sets can
be defined in the following way:
\begin{definition}
The $\omega-$\textit{limit set} of a point $A\in\mathcal{P}(V)$ under $\Phi_{G}%
$\ is defined as
\[
\omega(A)=\left\{  y\in V,%
\begin{array}
[c]{c}%
\text{there\ are\ }t_{k}\rightarrow\infty\text{ }\\
\text{such\ that }y\in\Phi(t_{k},A)
\end{array}
\right\}  \in\mathcal{P}(V)\text{.}%
\]

\end{definition}

\begin{remark}
Note that by definition of the discrete topology we have for $A\in\mathcal{P}(V)$
the fact $\omega(A)=\cup\{\omega(\{i\})$, $i\in A\}$.
\end{remark}

\begin{remark}
\label{l=limitnonempty}The existence of $\omega-$limit sets and the $L-$graph
property are closely related: Let $G$ be a graph and $\Phi_{G}$\ its
associated semiflow. Then $G$ is an $L-$graph iff $\omega(A)\neq\varnothing$
for all $A\in\mathcal{P}(V)$. This observation justifies our
concentration on $L-$graphs in this section.
\end{remark}

For continuous dynamical systems Morse decompositions are required
to contain all of the $\omega$, $\omega^{\ast}-$limit sets of the system.
For semiflows induced by graphs a weaker condition of
recurrence turns out to be appropriate:

\begin{definition}
\label{defrecur+}Let $G$ be an $L-$graph with associated semiflow
$\Phi_{G}$ on $\mathcal{P}(V)$. A one-point set $\{i\}\in
\mathcal{P}(V)$ is called recurrent, if there exists a sequence $n_{l}$ in
$\mathbb{N}$, $n_{l}\rightarrow\infty$, such that $\{i\}\subset\Phi_{G}%
(n_{l},\{i\})$. A set $B\in\mathcal{P}(V)$ is called recurrent if for each $i\in
B$ the one-point set $\{i\}$ is recurrent under $\Phi_{G}$. The set
$\mathcal{R}:=\{i\in V$, $\{i\}$ is recurrent$\}$ is called the recurrent set
of $\Phi_{G}$. If $\mathcal{R}=V$ the semiflow $\Phi_{G}$ is called recurrent.
\end{definition}

Note that by Definition \ref{defrecur+} it holds that $\{i\}\in
\mathcal{P}(V)$ is recurrent iff $i\in\omega(\{i\})$ iff $i\in\lambda$ for some
loop $\lambda$ of $G$. 
\vspace{0.07in}

\textbf{No-cycle condition:} For continuous dynamical systems the no-cycle
property  
(second item in Definition \ref{MorseDecomp}) is essential for the
characterization of a Morse decomposition via an order. 
 For semiflows induced by graphs we can formulate an
analogue of the no-cycle condition either using only the (forward) semiflow
$\Phi_{G}$\ or a combination of $\Phi_{G}$ and $\Phi_{G}^{-}$.

\begin{definition}
\label{defnocycle}Consider the semiflow $\Phi_{G}$ and a finite collection
$\mathcal{A}=\{A_{1},...,A_{n}\}$ of points in $\mathcal{P}(V)$. $\mathcal{A}$ is
said to satisfy the no-cycle condition for $\Phi_{G}$ if for any subcollection
$A_{j_{0}},...,A_{j_{l}}$ of $\mathcal{A}$ with $\omega(A_{j_{\alpha}})\cap
A_{j_{\alpha+1}}\neq\varnothing$ for $\alpha=0,...,l-1$ it holds that
$A_{j_{0}}\neq A_{j_{l}}$.
\end{definition}

\begin{remark}
Alternatively, we can define $A\in\mathcal{P}(V)$ to be a no-return set if for
all one-point sets $\{i\}\in\mathcal{P}(V)$ we have: If $\omega^{\ast}(\{i\})\cap
A\neq\varnothing$ and $\omega(\{i\})\cap A\neq\varnothing$ then $\{i\}\subset
A$, where $\omega^{\ast}(B)$ is the $\omega-$limit set for $B\in\mathcal{P}(V)$
under the negative semiflow $\Phi_{G}^{-}$.
\end{remark}
With these preparations we can now introduce our concept of a Morse
decomposition of the semiflow $\Phi_{G}$.

\begin{definition}
Let $G=(V,E)$ be an $L-$graph. A Morse decomposition of the semiflow $\Phi
_{G}$ on $\mathcal{P}(V)$ is a finite collection of nonempty,
pairwise disjoint and weakly invariant sets $\left\{  \mathcal{M}_{\mu}%
\in\mathcal{P}(V):\mu=1,...,k\right\}  $ such that:
\begin{itemize}
\item $\mathcal{R\subset\cup
}_{\mu=1}^{k}\mathcal{M}_{\mu}$ 

\item  $\left\{  \mathcal{M}_{\mu}%
\in\mathcal{P}(V):\mu=1,...,k\right\}  $ satisfies the no-cycle condition from
Definition \ref{defnocycle}.
\end{itemize}
 The elements of a Morse decomposition are called Morse sets.
\end{definition}

\begin{proposition}
\label{Morseorder}Let $G=(V,E)$ be an $L-$graph and let $$\mathcal{M}=\left\{
\mathcal{M}_{\mu}\in\mathcal{P}(V):\mu=1,...,k\right\} $$ be a finite collection
of nonempty, pairwise disjoint and weakly invariant sets of the semiflow
$\Phi_{G}$ on $\mathcal{P}(V)$. The collection $\mathcal{M}$ is a
Morse decomposition of $\Phi_{G}$ iff the following properties hold:

\begin{itemize}
\item $\mathcal{R\subset\cup}_{\mu=1}^{k}\mathcal{M}_{\mu}. $ 

\item The relation
\textquotedblleft$\preceq$\textquotedblright\ defined by
\[
\mathcal{M}_{\alpha}\,\preceq\mathcal{M}_{\beta}\text{ if}%
\begin{array}
[c]{c}%
\text{there\ are }\mathcal{M}_{j_{0}}=\mathcal{M}_{\alpha},\mathcal{M}_{j_{1}%
},...\mathcal{M}_{j_{l}\text{ }}=\mathcal{M}_{\beta}\text{ in }\mathcal{M}%
\text{ }\\
\text{with }\omega(\mathcal{M}_{j_{i}})\cap\mathcal{M}_{j_{i+1}}%
\neq\varnothing\text{ for }i=0,...,l-1
\end{array}
\]
is a (partial) order on $\mathcal{M}$.

\end{itemize}

 We use the indices $\mu=1,...,k$ in such a way that they reflect
this order, i.e. if $\mathcal{M}_{\alpha}\,\preceq\mathcal{M}_{\beta}$ then
$\alpha\leq\beta$.
\end{proposition}

\begin{proof}
Assume first that $\mathcal{M}=\left\{  \mathcal{M}_{\mu}\in\mathcal{P}(V)%
:\mu=1,...,k\right\}  $ is a Morse decomposition, we need to show that the
relation \textquotedblleft$\preceq$\textquotedblright\ is an order. Reflexivity: Let $\mathcal{M}_{\mu}\in\mathcal{M}$, then $\mathcal{M}_{\mu}$
is weakly invariant, i.e. $\Phi_{G}(n,A)\cap A\neq\varnothing$ for all
$n\in\mathbb{N}$. Since $\mathcal{M}_{\mu}$ consists of finitely many elements
there is at least one $i\in\mathcal{M}_{\mu}$ such that $i\in\Phi_{G}%
(n_{k},\mathcal{M}_{\mu})\cap\mathcal{M}_{\mu}\neq\varnothing$ for infinitely
many $n_{k}\in\mathbb{N}$. Hence $i\in\omega(\mathcal{M}_{\mu})\cap
\mathcal{M}_{\mu}$ and therefore $\omega(\mathcal{M}_{\mu})\cap\mathcal{M}%
_{\mu}\neq\varnothing$, which shows $\mathcal{M}_{\mu}\preceq\mathcal{M}_{\mu
}$. Antisymmetry: Assume there are $\mathcal{M_{\alpha}}$, $\mathcal{M}%
_{\beta}\in\mathcal{M}$ with $\mathcal{M}_{\alpha}\preceq\mathcal{M}_{\beta}$
and $\mathcal{M}_{\beta}\preceq\mathcal{M}_{\alpha}$. This means there\ are
$\mathcal{M}_{j_{0}}=\mathcal{M}_{\alpha},\mathcal{M}_{j_{1}},...\mathcal{M}%
_{j_{l}\text{ }}=\mathcal{M}_{\beta}$ in $\mathcal{M}$ with $\omega
(\mathcal{M}_{j_{i}})\cap\mathcal{M}_{j_{i+1}}\neq\varnothing$ for
$i=0,...,l-1$ and $\mathcal{M}_{k_{0}}=\mathcal{M}_{\beta},\mathcal{M}_{k_{1}%
},...\mathcal{M}_{k_{m}\text{ }}=\mathcal{M}_{\alpha}$ in $\mathcal{M}$ with
$\omega(\mathcal{M}_{j_{i}})\cap\mathcal{M}_{j_{i+1}}\neq\varnothing$ for
$i=0,...,m-1$. If there were two different sets in the collection
$\mathcal{M}_{j_{0}}=\mathcal{M}_{\alpha},\mathcal{M}_{j_{1}},...\mathcal{M}%
_{j_{l}\text{ }}=\mathcal{M}_{\beta}$, $\mathcal{M}_{k_{0}}=\mathcal{M}%
_{\beta},\mathcal{M}_{k_{1}},...\mathcal{M}_{k_{m}\text{ }}=\mathcal{M}%
_{\alpha}$, then $\mathcal{M}_{\alpha}\neq\mathcal{M}_{\alpha}$, which cannot
hold. Hence all sets in this collection are the same, in particular
$\mathcal{M}_{\alpha}=\mathcal{M}_{\beta}$. Transitivity: This follows
directly from the definition of \textquotedblleft$\preceq$\textquotedblright.

 Assume now that $\mathcal{M}=\left\{  \mathcal{M}_{\mu}%
\in\mathcal{P}(V):\mu=1,...,k\right\}  $ is a finite collection of nonempty,
pairwise disjoint and weakly invariant sets such the relation
\textquotedblleft$\preceq$\textquotedblright\ is an order. We need to show
that $\mathcal{M}$ satisfies the no-cycle condition: If $\mathcal{M}_{j_{0}%
},...,\mathcal{M}_{j_{l}}$ is a subcollection of $\mathcal{A}$ with
$\omega(A_{j_{\alpha}})\cap A_{j_{\alpha+1}}\neq\varnothing$ for
$\alpha=0,...,l-1$ then $\mathcal{M}_{j_{0}}\preceq\mathcal{M}_{j_{l}}$. If
this subcollection is disjoint, then $\mathcal{M}_{j_{0}}\npreceq
\mathcal{M}_{j_{l}}$, in particular $\mathcal{M}_{j_{0}}\neq\mathcal{M}%
_{j_{l}}$.
\end{proof}

As in the case of continuous dynamical systems, Morse decompositions
for semiflows induced by $L-$graphs need not be unique. For instance, the
collection $\{V,\varnothing\}$ always is a Morse decomposition of any
$\Phi_{G}$. As is the case of continuous dynamical systems we can use
intersections of Morse decompositions to refine existing ones, compare
page 8 in \cite{AH06}. Since the sets $V$ and $\mathcal{P}(V)$ are finite, the semiflow $\Phi_{G}$ on $\mathcal{P}(V)$ admits
a (unique) finest Morse decomposition for any $L-$graph $G$. The next
result characterises the finest Morse decomposition.

\begin{theorem}
\label{Morse=cc}Let $G=(V,E)$ be an $L-$graph with associated
semiflow $\Phi_{G}$ on $\mathcal{P}(V)$. For a finite
collection of
nonempty, pairwise disjoint sets $$\mathcal{M}=\left\{  \mathcal{M}_{\mu}%
\in\mathcal{P}(V):\mu=1,...,k\right\}$$ the following statements are
equivalent:

\begin{itemize}
\item $\mathcal{M}$ is the finest Morse decomposition of $\Phi_{G}$.

\item $\mathcal{M}=\mathcal{C}$, the set of communicating classes of $G$,
compare Lemma \ref{ccorder}.
\end{itemize}
\end{theorem}

\begin{proof}
Assume first that $\mathcal{M}=\left\{  \mathcal{M}_{\mu}\in\mathcal{P}(V)%
:\mu=1,...,k\right\}  $ is the finest Morse decomposition of
$\Phi_{G}$ and let $x,y\in\mathcal{M}_{\mu}$ for some $\mu=1,...,k$.
We have to show that $x$ and $y$ communicate. Observe first that
$x\in\mathcal{M}_{\mu}$ implies that there exists a loop $\gamma$ of
the graph $G$ with $x\in\gamma$. If
there is no such loop then $\{\mathcal{M}_{\mu}\backslash\{x\},\mathcal{M}%
_{\alpha}:\alpha\neq\mu\}$ is still a Morse decomposition and hence
$\mathcal{M}$ cannot be the finest one. If $x$ does not communicate
with $y$, take the communicating classes $[x]\neq\varnothing$ and
$[y]\neq\varnothing$, $[x]\cap\lbrack y]=\varnothing$ together with
set $\mathcal{L}=\{\lambda: \lambda\,\, \mbox{is a loop not in}\,
[x]\cup[y]\}$, to form the new Morse decomposition
${\mathcal{M}}^{\prime}=\{\mathcal{[}x],[y],\mathcal{L},\mathcal{M}_{\alpha}\,%
:\,\alpha\neq\mu\}$, which is finer than the given $\mathcal{M}$,
leading to a contradiction.

To see the converse, let $\mathcal{C}=\{C_{1},...,C_{k}\}$
be the set of communicating classes of the graph $G$. The
$C_{\alpha}$ are clearly nonempty, pairwise disjoint and weakly
invariant for all $\alpha=1,...,k$. Recall that
$\{i\}\in\mathcal{P}(V)$ is recurrent iff $i\in\lambda$ for some loop
$\lambda$ in $G$, and therefore we have
$\mathcal{R\subset\cup}_{\mu=1}^{k}C_{\mu}$. Finally, Lemma
\ref{ccorder} shows that the relation
\textquotedblleft$\preceq$\textquotedblright\ defined in Proposition
\ref{Morseorder} is indeed an order relation. Hence, the two ordered
sets $(\mathcal{C},\preceq)$ and $(\mathcal{M},\preceq)$ agree.
\end{proof}

\begin{remark}
\label{Morse=loop}The proofs of Proposition\ \ref{Morseorder} and of Theorem
\ref{Morse=cc} show the relationship between $\omega-$limit sets of $\Phi_{G}$
and loops of $G$: For each $A\in\mathcal{P}(V)$ the limit set $\omega(A)$
contains at least one loop. And vice versa, if $i\in\lambda$ is a vertex of a
loop $\lambda$ of $G$, then $i\in\omega(\{i\})$. This shows that for the
finest Morse decomposition $\mathcal{M}=\left\{  \mathcal{M}_{\mu}\text{, }%
\mu=1,...,k\right\}  $ of $\Phi_{G}$\ we have:
\[
\cup_{\mu=1}^{k}\mathcal{M}_{\mu}=\{i\in\lambda,\lambda\text{ is a loop of
}G\}\subset\cup\{\omega(A),A\in\mathcal{P}(V)\}\text{.}%
\]
This situation is different from the one for continuous dynamical systems,
where Morse sets may contain points that are not contained in limit sets, see
\cite{AH06}, Example 5.11. Indeed, it is not hard to see that for discrete
semiflows not all points in $\omega-$limit sets need to be elements of a Morse set.
\end{remark}

\begin{remark}
\label{Morse&omegalimit}Consider an $L-$graph $G$ with associated semiflow
$\Phi_{G}$. It follows from Theorem \ref{Morse=cc} that
$i\in\cup\{\omega(A)$, $A\in\mathcal{P}(V)\}$ $\backslash$ $\cup\{\mathcal{M}%
_{\mu}$, $\mathcal{M}_{\mu}$ is a finest Morse set$\}$ iff $i$ is a transitory
vertex and there exists a (finest) Morse set $\mathcal{M}$ with $i\in\Phi
_{G}(n,\mathcal{M})$ for some $n\geq1$.
\end{remark}

\subsection{Attractors and recurrence in semiflows}\label{arsemiflow}

Next we adapt the concept of an attractor to the semiflow on an $L-$graph
and analyse the connection with Morse decompositions.

\begin{definition}
\label{defattract+}Let $G=(V,E)$ be an $L-$graph with associated semiflow
$\Phi_{G}$ on $\mathcal{P}(V)$. A point $A\in\mathcal{P}(V)$ is
called an attractor if there exists a set $N\subset V$ with $A\subset N$ such
that $\omega(N)=A$.
\end{definition}

A set $N$ as in Definition \ref{defattract+} is called an
\textit{attractor neighborhood}. Note that in the discrete topology $A$ is a
neighborhood of itself and hence a point $A\in\mathcal{P}(V)$ is an attractor iff
$\omega(A)=A$. We also allow the empty set as an attractor.

A definition of repellers for semiflows of graphs is not obvious,
but the idea of complementary repellers from 
\cite{AH06} carries over with an obvious modification for semiflows:

\begin{definition}
For an attractor $A\in\mathcal{P}(V)$, the set
\[
A^{\ast}=\left\{  i\in V,\;\omega(\{i\})\backslash A\neq\varnothing\right\}
\in\mathcal{P}(V)%
\]
is called the complementary repeller of $A$, and $(A,A^{\ast})$ is called an
attractor-repeller pair.
\end{definition}

 Morse decompositions of semiflows can be characterized by
attractor-repellers pairs, in analogy to 
\cite{AH06}.
\begin{theorem}
\label{attractorMorsegraph}Let $G=(V,E)$ be an $L-$graph with associated
semiflow $\Phi_{G}$ on $\mathcal{P}(V)$. A finite collection of
sets $\mathcal{M}=\left\{  \mathcal{M}_{\mu}\in\mathcal{P}(V):\mu
=1,...,k\right\}  $ defines a Morse decomposition of $\Phi_{G}$ if and only if
there is a strictly increasing sequence of attractors
\[
\varnothing=A_{0}\subset A_{1}\subset A_{2}\subset...\subset A_{n}\subset V,
\]
such that
\[
\mathcal{M}_{n-i}=A_{i+1}\cap A_{i}^{\ast}\text{ for }0\leq i\leq n-1\text{.}%
\]

\end{theorem}

\begin{proof}
Recall the indexing convention for Morse sets from Proposition
\ref{Morseorder}.

 Let $\mathcal{M}=\left\{  \mathcal{M}_{\mu}\in\mathcal{P}(V)%
:\mu=1,...,k\right\}  $ be a Morse decomposition of $\Phi_{G}$. In analogy to
the continuous time case we define the sets $A_{k}$ for $k=1,...,n$ as
follows:%
\[
A_{k}=\{x\in V\text{, }\omega^{\ast}(x)\cap(\mathcal{M}_{n}\cup...\cup
\mathcal{M}_{n-k+1})\neq\varnothing\}\text{.}%
\]
Note that for semiflows of graphs we have $A_{k}=\mathcal{O}^{+}%
(\mathcal{M}_{n}\cup...\cup\mathcal{M}_{n-k+1})$. We first need to show that
each $A_{k}$ is an attractor. The inclusion $\omega(A_{k})\subset A_{k}$
follows directly from the characterization above of A$_{k}$ as a positive
orbit. To see that $A_{k}\subset\omega(A_{k})$ pick $x\in A_{k}$. Then there
exists $\mu\in\{n-k+1,...,n\}$ with $x\in\mathcal{O}^{+}(\mathcal{M}_{\mu})$.
According to Theorem \ref{Morse=cc} the set $\mathcal{M}_{\mu}$ is a
communicating class of the graph $G$ and hence every element
$z\in\mathcal{M}_{\mu}$ is in a loop $\gamma$ that is completely contained in
$\mathcal{M}_{\mu}$, compare Lemma \ref{ccloop}. Therefore there is $z\in
A_{k}$ with $x\in\Phi_{G}(n_{l},\{z\})$ for as sequence $n_{l}\rightarrow
\infty$, i.e. $A_{k}\subset\omega(A_{k})$. Hence each $A_{k}$ is an attractor.

 Next we show that $\mathcal{M}_{n-i}=A_{i+1}\cap A_{i}^{\ast}$ for
$0\leq i\leq n-1$. To see that $\mathcal{M}_{n-i}\subset A_{i+1}$ pick
$x\in\mathcal{M}_{n-i}$. Since $\mathcal{M}_{n-i}$ is a communicating class of
the graph $G$ we have $\omega^{\ast}(x)\cap\mathcal{M}_{n-i}\neq\varnothing$
and therefore $x\in A_{i+1}$. To see that $\mathcal{M}_{n-i}\subset
A_{i}^{\ast}$ assume that there exists $x\in\mathcal{M}_{n-i} $ with $x\notin
A_{i}^{\ast}$, i.e. $\omega(x)\backslash A_{i}=\varnothing$ or $\omega
(x)\subset A_{i}$. But $x\in\mathcal{M}_{n-i} $ means $\omega(x)\cap
\mathcal{M}_{n-i}\neq\varnothing$, and by definition we have $\mathcal{M}%
_{n-i}\cap A_{i}=\varnothing$, which is a contradiction. This shows
$\mathcal{M}_{n-i}\subset A_{i+1}\cap A_{i}^{\ast} $ for $0\leq i\leq n-1$. To
see the reverse inclusion, let $x\in A_{i+1}\cap A_{i}^{\ast}$, i.e.
$x\in\mathcal{O}^{+}(\mathcal{M}_{n}\cup...\cup\mathcal{M}_{n-i})$ and
$\omega(x)\backslash A_{i}\neq\varnothing$. Recall that by Remark
\ref{Morse=loop} $\omega(x)$ contains a loop $\gamma$ of the graph $G$, and by
Lemma \ref{ccloop} and Theorem \ref{Morse=cc} each loop is contained in a
Morse set. Hence $\omega(x)\cap(\mathcal{M}_{n}\cup...\cup\mathcal{M}%
_{n-i})\neq\varnothing$, which by definition of a Morse decomposition means
that $x\in\mathcal{M}_{n-i}$.

 Let $\mathcal{M}_{n-i}=A_{i+1}\cap A_{i}^{\ast}$ for $0\leq
i\leq n-1$ be defined as in the statement of the theorem. We have to show that
$\{\mathcal{M}_{1},...,\mathcal{M}_{n}\}$ form a Morse decomposition. We start
by proving that the sets $\mathcal{M}_{n-i}=A_{i+1}\cap A_{i}^{\ast}$ are
nonempty. Note first of all that $A_{1},...,A_{n}\neq\varnothing$ by
assumption. We have by definition of attractor-repeller pairs that
$V=A_{0}^{\ast}\supset A_{1}^{\ast}\supset...\supset A_{n-1}^{\ast}\supset
A_{n}^{\ast}$. Now $A_{n-1}^{\ast}\neq\varnothing$ can be seen like this: If
$A_{n-1}^{\ast}=\varnothing$ then for all $x\in V$ we have $\omega
(x)\backslash A_{n-1}=\varnothing$, i.e. $\omega(x)\subset A_{n-1}$. Hence
there is $m\in\mathbb{N}$ such that for all $\alpha\geq m$ we have $\Phi
_{G}(\alpha,V\backslash A_{n-1})\subset A_{n-1}$ and therefore $A_{n}$ cannot
be an attractor. We conclude that $A_{0}^{\ast},A_{1}^{\ast},...,A_{n-1}%
^{\ast}\neq\varnothing$. Now if $A_{i+1}\cap A_{i}^{\ast}=\varnothing$ then we
have by the same reasoning as before: For all $x\in A_{i+1}$ it holds that
$\omega(x)\backslash A_{i}=\varnothing$, i.e. $\omega(x)\subset A_{i}$ and
$A_{i+1}$ cannot be an attractor.

The sets $\mathcal{M}_{i}$ are pairwise disjoint: Let $\alpha<\beta
$, then $M_{n-\alpha}\cap M_{n-\beta}=A_{\alpha+1}\cap A_{\alpha}^{\ast}\cap
A_{\beta+1}\cap A_{\beta}^{\ast}=A_{\alpha+1}\cap A_{\beta}^{\ast}\subset
A_{\beta}\cap A_{\beta}^{\ast}=\varnothing$.

 The sets $\mathcal{M}_{i}$ are weakly invariant: As above, it
suffices to prove that $\mathcal{M}_{n-i}=A_{i+1}\cap A_{i}^{\ast}$ contains a
loop of the graph $G$. If there is no loop in $A_{i+1}\cap A_{i}^{\ast}$, then
there exists $m\in\mathbb{N}$ such that for all $\alpha\geq m$ we have
$\Phi_{G}(\alpha,A_{i+1}\cap A_{i}^{\ast})\subset A_{i}$ and therefore
$A_{i+1}$ cannot be an attractor.

The collection $\{\mathcal{M}_{1},...,\mathcal{M}_{n}\}$ satisfies the
no-cycle condition: This is just a restatement of the assumption that
$A_{0}\subset A_{1}\subset A_{2}\subset...\subset A_{n}$ is a strictly
increasing sequence of attractors.
 We have shown so far that $\mathcal{M}=\{\mathcal{M}_{1},...,\mathcal{M}_{n}\}$ satisfies the conditions of a Morse decomposition,
except for $\mathcal{R\subset\cup}_{\mu=1}^{n}\mathcal{M}_{\mu}$. Now let
$\mathcal{M}^{\prime}=\{\mathcal{M}_{1}^{\prime},...,\mathcal{M}_{k}^{\prime}\}$
be the finest Morse decomposition of $\Phi_{G}$. Since the recurrence
condition was not used in the first part of the proof of Theorem
\ref{Morse=cc}, and since $\{i\}\in\mathcal{P}(V)$ is recurrent iff $i\in\lambda$
for some loop $\lambda$ of $G$, we know by Lemma \ref{loopcc} that
$\mathcal{R\subset\cup}_{\mu=1}^{k}\mathcal{M}_{\mu}^{\prime}\subset
\mathcal{\cup}_{\mu=1}^{n}\mathcal{M}_{\mu}$. Altogether we see that
$\mathcal{M}=\{\mathcal{M}_{1},...,\mathcal{M}_{n}\}$ is a Morse decomposition.
\end{proof}

\begin{corollary}
\label{Morseset=attr}Let $\mathcal{M}=\left\{  \mathcal{M}_{\mu}\in
\mathcal{P}(V):\mu=1,...,k\right\}  $ be the finest Morse decomposition of a
semiflow $\Phi_{G}$ on $\mathcal{P}(V)$, with order $\preceq$.
Then the maximal (with respect to $\preceq$) Morse sets are attractors.
Furthermore, the smallest (with respect to set inclusion) non-empty attractors
are exactly the maximal (with respect to $\preceq$) Morse sets.
\end{corollary}

\begin{proof}
If $M$ is a maximal Morse set of the semiflow $\Phi_{G}$, then, according to
Theorem \ref{Morse=cc} and Proposition \ref{Morseorder}, $M$ is a maximal
communicating class of the graph $G$. Hence $M$ is forward invariant,
$\omega(M)=M$ and $M$ does not contain any attractor, except for the empty set.

 Vice versa, if $A$ is a smallest (with respect to set inclusion)
non-empty attractor, then $A$ is a Morse set according to Theorem
\ref{attractorMorsegraph}. If $A$ is not maximal (with respect to $\preceq$),
then $A$ is not forward invariant for the graph $G$ and hence there exists a
point $x\in\mathcal{O}^{+}(A)\backslash A$ such that $\mathcal{O}^{+}(x)\cap
A=\varnothing$ (by Lemma \ref{pathout}). According to Remark \ref{ficcinorbit}%
, $\mathcal{O}^{+}(x)$ contains a maximal communicating class, which is an
attractor $A^{\prime}\varsubsetneqq A$ and hence $A$ is not a smallest
non-empty attractor.
\end{proof}

 It remains to analyse the behaviour of the semiflow $\Phi_{G}$ on a
Morse set. Definition \ref{defrecur+} and Remark \ref{Morse=loop} already
point at a recurrence property that holds for $\omega-$limit sets: Note that
by Definition \ref{defrecur+} it holds: $\{i\}\in\mathcal{P}(V)$ is recurrent iff
$i\in\omega(\{i\})$ iff $i\in\lambda$ for some loop $\lambda$ of $G$. Hence we
obtain from Remark \ref{Morse=loop} for the finest Morse decomposition
$\mathcal{M}=\left\{  \mathcal{M}_{\mu}\text{, }\mu=1,...,k\right\}  $ of
$\Phi_{G}$%
\begin{equation}
\mathcal{R}=\cup_{\mu=1}^{k}\mathcal{M}_{\mu}\text{.}\label{rec=Morsesets}%
\end{equation}
The recurrent set is partitioned into the disjoint sets of the finest Morse
decomposition under the following natural concept of connectedness.

\begin{definition}
\label{semiflowconnect}A set $B\in\mathcal{P}(V)$ is called connected under
$\Phi_{G}$ if for any $i,j\in B$ there exist $n\in\mathbb{N}$ and a map
$p:\{0,...,n\}\rightarrow B$ with the properties

\begin{itemize}
\item $p(0)=i$, $p(n)=j$

\item $p(m+1)\in\Phi_{G}(1,\{p(m)\})$ for $m=0,...,n-1$.
\end{itemize}

 The flow $\Phi_{G}$ is called strongly connected if the set of
vertices $V$ is connected under $\Phi_{G}$.
\end{definition}

The following result then characterises the behaviour of the semiflow
$\Phi_{G}$ on its Morse sets, compare Theorem 6.4 in \cite{AH06} for continuous
dynamical systems.

\begin{theorem}
\label{attrrec}Let $G=(V,E)$ be an $L-$graph with associated semiflow
$\Phi_{G}$ on $\mathcal{P}(V)$. The recurrent set $\mathcal{R}$
of $\Phi_{G}$ satisfies
\[
\mathcal{R}=\bigcap\left\{  A\cup A^{\ast},\;\,A\text{ is\ an\ attractor}%
\right\}
\]
and the (finest) Morse sets of $\Phi_{G}$ coincide with the $\Phi_{G}%
$-connected components of $\mathcal{R}$.
\end{theorem}

\begin{proof}
Assume that $x\in\mathcal{R}$, then $x\in\gamma$ for some loop $\gamma$ of the
graph $G$. Let $A$ be an attractor for $\Phi_{G}$, then if $x\in A$ we are
done. Otherwise if $x\notin A$ then it holds that $\gamma\cap A=\varnothing$.
But $\gamma\subset\omega(x)$ and therefore $\omega(x)\backslash A\neq
\varnothing$, which means that $x\in A^{\ast}$. Conversely, if $x\in
\cap\{A\cup A^{\ast},\;\,A$ is\ an\ attractor$\}$, then $x$ is in any
attractor containing $\omega(x)$. Arguing as in the proof of Theorem
\ref{attractorMorsegraph}, there exists a loop $\gamma$ of the graph such that
$x\in\gamma$, which shows that $x\in\mathcal{R}$.

The second statement of the theorem follows directly from Definition
\ref{semiflowconnect} and (\ref{rec=Morsesets}).
\end{proof}

 As discussed in the paragraph about invariance (see Section \ref{semiMorsedecomp}), forward
invariance under $\Phi_{G}$ is a fairly strong requirement for a set $A\subset
V$, and thus it appears that there are few sets to which one can restrict the
semiflow $\Phi_{G}$, namely (unions of positive) orbits. 

\begin{definition}
Let $G=(V,E)$ be an $L-$graph with associated semiflow $\Phi_{G}$ on
$\mathcal{P}(V)$. Let $G^{\prime}=(V^{\prime},E^{\prime})$ be
the subgraph of $G$ for a subset of vertices $V^{\prime}\subset V$. The
resulting semiflow $\Phi_{G^{\prime}}$ on $\mathcal{P}%
(V^{\prime})$ is called the semiflow $\Phi_{G}$ restricted to $V^{\prime}$.
\end{definition}

 Note that if $\mathcal{M}\subset V$ is a Morse set of $\Phi_{G}$,
then the induced graph $(V_{\mathcal{M}}, E_{\mathcal{M}})$ is an $L-$graph. This observation allows us to prove the following fact about Morse sets and recurrence.

\begin{corollary}
Under the conditions of Theorem \ref{attrrec}\ the semiflow $\Phi_{G}$
restricted to any Morse set is recurrent.
\end{corollary}
\begin{proof}
 The proof follows directly from Theorem \ref{Morse=cc} and
Definition \ref{defrecur+}.
\end{proof}

 As we have seen, most of the concepts used to characterise the
global behaviour of continuous dynamical systems can be adapted in a natural
way to the positive semiflow of an $L-$graph, resulting in very similar
characterisations. Indeed, the proofs for semiflows on a finite set are
considerably simpler than the corresponding ones for continuous dynamical
systems. What is missing in the context of semiflows is first of all the group
property of a flow, and hence limit objects for $t\rightarrow-\infty$. This
results in missing some of the invariance properties of crucial sets, such as
limit sets, Morse sets, the (components of) the recurrent set, etc. And
secondly, the use of the discrete topology implies that while all points in
the (finest) Morse sets are limit points, not all $\omega-$limit points of the
semiflow are contained in the (finest) Morse sets. But those exceptional limit
points (and hence the set of all limit points) can be characterised, compare
Remark \ref{Morse&omegalimit}.


\section{Matrices associated with graphs and their semiflows\label{matrix}}

In the previous sections we have analyzed the communication
structure of graphs using two different mathematical languages, that
of graph theory and that of dynamical systems. In this section we
will briefly use yet another language, matrices and linear algebra.
Connections between graphs and nonnegative matrices have been
studies extensively in the literature, compare \cite{BNS89, BP94, R06}. We
will describe some connections and hint at algorithms that allow for
the computation of the objects discussed in the previous sections.

\begin{definition}
\label{defadjmatrix}Given a graph $G=(V,E)$ with $\#V=d$. The
adjacency matrix
$A_{G}=(a_{ij})$ of $G$ is the $d\times d$ matrix with elements%
\[
a_{ij}=\left\{
\begin{array}
[c]{rl}%
1 & \text{if}\,(i,j)\in E\\
0 & \text{otherwise.}%
\end{array}
\right.
\]

\end{definition}

\noindent Vice versa, denote the set of $d\times d$ matrices whose
entries are in $\{0,1\}$ by $M(d,\{0,1\})$. Any matrix $A\in
M(d,\{0,1\})$ is called an adjacency matrix and can be viewed as
representing a graph. 

\begin{definition}
\label{deflmatrix}An adjacency matrix is called an $L-$matrix if
each row has at least one entry equal to $1$.
\end{definition}

\begin{remark}
\label{path&product}Alternatively, the entries $a_{ij}\in A_{G}$ can
be viewed as paths in $G$ of length $1$: $a_{ij}=1$ iff there exists
an edge $\gamma \in\Gamma^{1}$ with $\pi_{0}(\gamma)=i$ and
$\pi_{1}(\gamma)=j$. Continuing this thought we have the following
relationship between paths of length $n\geq1$ and entries of
$A_{G}^{n}$, the $n-$th power of $A_{G}$: $a_{ij}^{(n)}\in
A_{G}^{n}$ is exactly the number of (different) paths
$\gamma\in\Gamma^{n}$ from $i$ to $j$ in $G$. Hence the $i-$th row
of $A_{G}^{n}$ describes exactly the vertices that can be reached
from $i$ via a path of length $n$. In complete analogy, the $i-$th
row of $(A_{G}^{T})^{n}$ describes exactly the vertices from which
$i$ can be reached using a path of length $n$. 
\end{remark}

\noindent The communication concepts developed in Section \ref{orbits&cc} only depend on the existence of
paths connecting certain vertices and not on the number of such
paths. We use the standard Boolean addition $+^{\ast}$\ and multiplication
$\cdot^{\ast}$\ to describe these ideas. We extend the Boolean addition and multiplication for matrices in $M(d,\{0,1\})$ in the obvious way,
denoting by $A^{n\ast}$ the $n-$th Boolean product of $A\in
M(d,\{0,1\})$ with itself. Note that $M(d,\{0,1\})$ is closed under
Boolean addition and multiplication. Since computer calculations
involving $+^{\ast}$ and $\cdot^{\ast}$\ are very fast, we obtain
efficient algorithms for the computation of orbits and communicating
classes.

\noindent Let $G=(V,E)$ be a graph with $\#V=d$ and adjacency matrix
$A$. For
a vertex $i\in V$ its positive and negative orbits are given by%
\begin{align*}
O^{+}(i)  & =\{j\in V\text{, }A_{ij}^{n\ast}=1\text{ for some }n=1,...,d\}\\
O^{-}(i)  & =\{j\in V\text{, }(A^{T})_{ij}^{n\ast}=1\text{ for some
}n=1,...,d\}\text{.}%
\end{align*}
The proof of these facts follows directly from Section
\ref{orbits&cc} and Remark \ref{path&product} above. The
communicating classes of $G$ can now be computed as
$C=\mathcal{O}^{+}(i)\cap\mathcal{O}^{-}(i)$ for $i\in V$, compare
Theorem \ref{cc=orbit}. The order among communicating classes can be
determined directly from the computation of $O^{+}(i)$ (or
$O^{-}(i)$), compare Remark \ref{ccinorbit}\ and Definition
\ref{deforder}. Communicating classes and their order are also
sufficient to compute the quotient graphs $G_{q}$ and $G_{Q}$ of a
given graph $G$. Hence all the objects analyzed in Section
\ref{orbits&cc} can be computed effectively using the matrix ideas
described above.

\noindent The concepts of irreducibility and aperiodicity play an
important role in the analysis of nonnegative matrices. We briefly
introduce these concepts and discuss their use in the analysis of
communication structures in graphs.

\begin{definition}
\label{defmatrixirred}A matrix $A\in M(d,\{0,1\})$ is said to be
irreducible
if it is not permutation similar to a matrix having block-partition form%
\[%
\begin{pmatrix}
A_{11} & A_{12}\\
0 & A_{22}%
\end{pmatrix}
\]
with $A_{11}$ and $A_{22}$ square.
\end{definition}

\begin{lemma}
Let $G$ be a graph with adjacency matrix $A$. Then $A$ is
irreducible iff $G$ consists of exactly one communicating class.
\end{lemma}

\noindent A proof of this lemma can be found, e.g., in \cite{BP94}.
Note that the adjacency matrix of any $L-$graph is permutation
equivalent to a matrix of
the form%
\begin{equation}
A_{Q}=%
\begin{pmatrix}
A_{11} & A_{12} & \cdot & \cdot & \cdot & A_{1l}\\
0 & A_{22} &  &  &  & A_{2l}\\
\cdot &  & \cdot &  &  & \cdot\\
\cdot &  &  & \cdot &  & \cdot\\
\cdot &  &  &  & \cdot & \cdot\\
0 & \cdot & \cdot & \cdot & 0 & A_{ll}%
\end{pmatrix}
\label{matrixcc}%
\end{equation}
where the square blocks $A_{ii}$, $i=1,...,l$ correspond to the
vertices within the communicating class $C_{i}$ for $i=1,...,k$ and
to the transitory vertices, and the blocks $A_{ij}$ for $j>i$
determine the order structure among the communicating classes. Hence
$A_{Q}$ ``is" the adjacency matrix of the extended quotient graph
$G_{Q}$, compare Section \ref{quograph}. For additional
characterizations of irreducible nonnegative matrices we refer to
\cite{R06}, Chapter 9.2, Fact 2.

\begin{definition}
\label{defmatrixperiod}Let $A\in M(d,\{0,1\})$ be irreducible with
associated graph $G$. The period of $A$ is defined to be the
greatest common divisor of the length of loops of $G$. If this
period is $1$, the matrix is said to be aperiodic. We say that a
graph is periodic of period $p$ (or aperiodic) if its adjacency
matrix has this property.
\end{definition}

\begin{lemma}
A matrix $A\in M(d,\{0,1\})$ is aperiodic iff $A^{n}>0$ (has all
elements $>0 $) for some $n\in\mathbb{N}$.
\end{lemma}

\noindent A proof of this lemma can be found, e.g., in \cite{BP94}.
For additional characterizations of aperiodic matrices we refer to
\cite{R06}, Chapter 9.2, Fact 3. One can extend this definition to
any communicating class of a graph $G$: Let $C\subset V$ be a
communicating class of $G$ and $A_{C}$ its diagonal block in the
representation (\ref{matrixcc}) of the adjacency matrix $A_{G}$.
Note that $A_{C}$ is irreducible and we define the period of $C$ to
be the period of $A_{C}$.

\noindent The rest of this section is devoted to studying some of
the connections between the semiflow of a graph and the adjacency
matrix. Since
the semiflow is a sequence of maps $\Phi_{G}(n,\cdot):\mathcal{P}%
(V)\rightarrow\mathcal{P}(V)$, $n\in\mathbb{N}$, we first need to
define the analogue of $\mathcal{P}(V)$. For a graph $G=(V,E)$ with
$\#(G)=d$, we can proceed as follows:

\noindent For a subset $A\subset V$ let $\chi_{A}$ denote its
characteristic
function, i.e.%
\[
\chi_{A}(i)=\left\{
\begin{array}
[c]{cc}%
1 & \text{ if }i\in A\\
0 & \text{if }i\notin A\text{.}%
\end{array}
\right.
\]
Let $e_{i}$ be the $i-th$ canonical basis vector of
$\mathbb{R}^{d}$. Define
$\iota:\mathcal{P}(V)\rightarrow\mathbb{R}^{d}$ by
\[
\iota(A)=\sum\nolimits_{i=1}^{d}\chi_{A}(i)e_{i}\text{.}%
\]
We denote by $\mathcal{Q}^{d}$ the (vertex set of the) unit cube in
$\mathbb{R}^{d}$. Note that
$\iota:\mathcal{P}(V)\rightarrow\mathcal{Q}^{d}$ is bijective and
hence we can identify $\mathcal{P}(V)$ with $\mathcal{Q}^{d}$ as
sets. We will use the same notation for the two versions of the map
$\iota$.

\noindent With these notations we can express the semiflow
$\Phi_{G}$ in terms of the adjacency matrix $A$: For paths of $G$ of
length $1$, i.e. for edges we
have $\mathcal{O}_{1}^{+}(i):=\{j\in V$, $(i,j)\in E\}=\iota^{-1}%
(\iota(\{i\})^{T}\cdot^{\ast}A)$. The set $\mathcal{O}_{1}^{+}(i)$
is also called the orbit of $i$ at time $1$. Similarly we have for
$W\subset V$: $\mathcal{O}_{1}^{+}(W):=\{j\in V$, $(i,j)\in E$ for
some $i\in W\}=\iota ^{-1}(\iota(W)^{T}\cdot^{\ast}A)$. By Remark
\ref{path&product} we obtain for
paths of length $n\geq1$%
\[
\mathcal{O}_{n}^{+}(W):=\left\{
\begin{array}
[c]{cc}%
j\in V\text{,} & \text{there are }i\in W\text{ and
}\gamma\in\Gamma^{n}\text{
such}\\
& \text{that }\pi_{0}(\gamma)=i\text{ and }\pi_{n}(\gamma)=j
\end{array}
\right\}  =\iota^{-1}(\iota(W)^{T}\cdot^{\ast}A^{n\ast})\text{.}%
\]
The definition of the associated (positive) semiflow $\Phi_{G}$ in
Equation (\ref{defflow+}) now yields the following alternative way
of describing this
semiflow%
\begin{equation}
\Phi_{G}(n,W)=\iota^{-1}((\iota(W)^{T}\cdot^{\ast}A^{n\ast})^{T}%
)\text{.}\label{flowident}%
\end{equation}
This observation justifies the following definition:

\begin{definition}
\label{defmatrixflow}Consider a matrix $A\in M(d,\{0,1\})$ and let
$\mathcal{Q}^{d}\subset\mathbb{R}^{d}$ be the $d-$dimensional unit
cube. The
map%
\[
\Psi_{A}:\mathbb{N}\times\mathcal{Q}^{d}\rightarrow\mathcal{Q}^{d}\text{,}%
\]%
\[
\Psi_{A}(n,q)=(q^{T}\cdot^{\ast}A^{n\ast})^{T}%
\]
is called the positive semiflow of $A$. Similarly, the map
\[
\Psi_{A}^{-}:\mathbb{N}^{-}\times\mathcal{Q}^{d}\rightarrow\mathcal{Q}%
^{d}\text{,}%
\]%
\[
\Psi_{A}^{-}(n,q)=(q^{T}\cdot^{\ast}(A^{T})^{n\ast})^{T}%
\]
is called the negative semiflow of $A$.
\end{definition}

\noindent It follows from (\ref{flowident}), Proposition
\ref{semiflow} and from bijectivity of
$\iota:\mathcal{P}(V)\rightarrow\mathcal{Q}^{d}$ that $\Psi_{A}$ and
$\Psi_{A}^{-}$ are, indeed, semiflows. This allows us to reinterpret
all concepts and results from Sections \ref{graphsemiflow} -
\ref{arsemiflow} for semiflows of square $\{0,1\}-$matrices.
Alternatively, we could have developed the theory for semiflows of
the type $\Psi_{A}$ and then translated the results to graphs. The
key condition in Sections \ref{graphsemiflow}-\ref{arsemiflow} is
for a graph to be an $L-$graph, which translates into $L-$matrices,
see Definition \ref{deflmatrix}.

\noindent To complete this chapter, we mention a few connections
between the semiflow $\Phi_{G}$ of an $L-$graph and concepts from
matrix theory:

\noindent Let $G=(V,E)$ be an $L-$graph with associated semiflow
$\Phi:\mathbb{N}\times\mathcal{P}(V)\rightarrow\mathcal{P}(V)$.
Let $A$ be
the adjacency matrix of $G$ with associated semiflow $\Psi_{A}:\mathbb{N}%
\times\mathcal{Q}^{d}\rightarrow\mathcal{Q}^{d}$. Then it holds:

\begin{enumerate}
\item $G$ has exactly one communicating class iff $\Phi$ has only the trivial
Morse decomposition $\{\varnothing,V\}$ iff $A$ is irreducible.

\item A vertex $i\in V$ is in a communicating class $C\subset V$ iff $\{i\}$
is recurrent under $\Phi$ iff
$(\sum\nolimits_{n=1}^{d-1}A^{n\ast})_{ii}>0$.

\item For two communicating classes $C_{\mu}\preceq C_{\nu}$ holds iff their
corresponding Morse sets satisfy
$\mathcal{M}_{\mu}\preceq\mathcal{M}_{\nu}$ iff the adjacency matrix
$A_{\mu\mu}$ of the subgraph corresponding to $C_{\mu}$ has a
smaller index in the representation (\ref{matrixcc}) than
$A_{\nu\nu}$.
\end{enumerate}

\noindent In the next chapter we will study Markov chains and
interpret many of the results we have obtained so far in that
context. We prefer to include some elementary results and definitions on Markov chains to facilitate exposition.


\section{Characterization of Markov Chains via Graphs and Semiflows}

Markov chains are discrete time stochastic processes for which the
future is conditionally independent of the past, given the presence. If the state space of a Markov chain is a finite set, its
probabilistic behavior can be analyzed using specific graphs and/or matrices. The goal of this section is to utilize the concepts and theory developed in Section 3 for the analysis of finite state Markov chains. This allows us to restate some well-known properties of Markov chains using graphs and semiflows, and to show a few new connections. We end up this work by setting up a \textquotedblleft three-column dictionary\textquotedblright of equivalent objects and results in the three \textquotedblleft languages\textquotedblright.

\subsection{Review of Markov chains}

In this section we present without proofs some standard results in the theory of Markov chains. We refer the reader
to \cite{M94, S99, TK94} for more details on
finite chains, and to \cite{AL06} Chapter 14.1, for a thorough
discussion of chains on countable state spaces. 

\noindent Let $(\Omega,\mathcal{F},\mathbb{P})$ be a probability
space and $S$ a finite set with cardinality $\#(S)=d$. 
A discrete
time stochastic process on $\Omega$ with values in the state space
$S$ is a sequence of random variables $X_{n}:\Omega\rightarrow S$,
$n\in\mathbb{N}$. For each $\omega\in\Omega$ the sequences
$(X_{n}(\omega)$, $n\in\mathbb{N)}$ are called trajectories of the
process. Recall that a Markov chain with values in $S$ is a discrete time stochastic
process satisfying the Markov property, i.e.
\[
\mathbb{P}\{X_{n+1}=j\mid X_{0}=i_{0},...,X_{n-1}=i_{n-1},X_{n}=i\}=\mathbb{P}%
\{X_{n+1}=j\mid X_{n}=i\}
\]
for all times $n\in\mathbb{N}$ and all states
$i_{0},...,i_{n-1},i,j\in S$.

Recall the so-called Chapman-Kolmogorov
equation for the
n-step transition probabilities is given by:
\begin{equation}
p_{n}(i,j)=\mathbb{P}\{X_{n}=j\mid X_{0}=i\}=\underset{k\in S}{\sum}%
p_{r}(i,k)p_{n-r}(k,j)\label{defCKequation}%
\end{equation}
for $1<r<n$.

%
%
%
\begin{definition}
\label{defmccommun}A state $i\in S$ has access to a state $j\in S$
if $p_{n}(i,j)>0$ for some $n\geq0$. A state $i\in S$ communicates
with a state $j\in S$ if $p_{n}(i,j)>0$ and $p_{m}(j,i)>0$ for some
$n,m\geq0$.
\end{definition}
%
%
%
%
\begin{definition}
\label{defmcperiod}For a state $i\in S$ we define its period by
\[
\delta(i)=\gcd\{n\geq1\text{, }p_{n}(i,i)>0\}\text{,}%
\]
where $\gcd$ denotes the greatest common divisor. Then $i$ is called
periodic if $\delta(i)>1$, and aperiodic if $\delta(i)=1$. A Markov
chain is said to be aperiodic if all points $i\in S$ are aperiodic.
\end{definition}
%

\noindent A crucial idea in the analysis of the qualitative behavior
of stochastic processes is that of reachability, i.e., trajectories
starting at one point reach (a neighborhood of) another point. For
Markov chains this idea takes the form of hitting times: For
$A\subset S$ we define the first hitting
time of $A$ as the random variable%
\[
\tau_{A}(\omega):=\inf\{n\geq1\text{, }X_{n}(\omega)\in A\}\text{,}%
\]
with the understanding that $\tau_{A}(\omega)=\infty$ if the $\inf$
does not exist. We often drop the variable $\omega$ and denote the
(conditional)
distributions of the first hitting times by%
\[
f_{n}(i,A):=\mathbb{P}\{\tau_{A}=n\mid X_{0}=i\}\text{.}%
\]
Using this idea we can generalize the Chapman-Kolmogorov equation to random intermediate time points as follows
(recall
that we use the notation $P^{0}=I$)%
\begin{equation}
p_{n}(i,j)=\underset{r=1}{\overset{n}{\sum}}f_{r}(i,j)p_{n-r}(j,j)\text{
\ \ for }n\geq1\text{.}\label{randomCKequation}%
\end{equation}
In particular, the stochastic process with transition probability
matrix $P$ and initial variable $\tau_{A}$ is again a Markov chain
for any $A\subset S$.

\noindent The idea of a first hitting time leads to the definition
of a sequence of random variables of subsequent visits to a state or
set of states: Define
$\tau_{A}^{(m+1)}(\omega):=\inf\{n>\tau_{A}^{(m)}(\omega) $,
$X_{n}(\omega)\in A\}$ for $m\geq1$, with
$\tau_{A}^{(1)}(\omega)=\tau _{A}(\omega)$. With this notation we
can write the number of visits of a set
$A\subset S$ up to time $m\geq1$ as%
\[
N_{A}(m,\omega):=\underset{n=1}{\overset{m}{\sum}}\chi_{A}(X_{n}%
(\omega))\text{,}%
\]
and the total number of visits as
\[
N_{A}(\omega):=\underset{n\geq1}{\sum}\chi_{A}(X_{n}(\omega))=\#\{m\geq
1\text{, }\tau_{A}^{(m)}<\infty\}\text{,}%
\]
where $\chi_{A}$ denotes again the characteristic function of a set
$A$.

\noindent Two other concepts derived from first hitting times play a
role in the analysis of Markov chains: The first one is the
(conditional) probability of reaching a set of states $A$ from a
state $i$, i.e.
\[
\mathbb{P}\{\omega\in\Omega,X_{n}(\omega)\in A\text{ for some
}n\geq1\mid
X_{0}=i\}=\underset{n\geq1}{\sum}f_{n}(i,A)=:F(i,A)\text{.}%
\]
The other useful probabilistic concept is that of moments, where we
will use
only the first moment, i.e. the mean first hitting time%
\[
\mu(i,A):=\underset{n\geq1}{\sum n}f_{n}(i,A)\text{.}%
\]

\noindent With these preparations we can define the ideas of
recurrence and transience. When talking about points $j\in S$ we
often use the notation $f_{n}(i,j)$, $F(i,j)$, $\mu(i,j)$ instead of
$f_{n}(i,\{j\})$ etc.

\begin{definition}
\label{defrecurmc}A state $i\in S$ is called recurrent, if $\mathbb{P}%
\{\tau_{i}<\infty\mid X_{0}=i\}=F(i,i)=1$. States that are not
recurrent are called transient.

\noindent A recurrent state $i\in S$ satisfying
$\mathbb{E}(\tau_{i}\mid X_{0}=i)=\mu(i,i)<\infty$ is called
positive recurrent, the other recurrent states are called null
recurrent. Here we denote by $\mathbb{E}(Y\mid X_{0}=i)$ the
conditional expectation of a random variable $Y$ under the measure
$\mathbb{P}(\cdot\mid X_{0}=i)$.
\end{definition}

\noindent We list some standard results regarding the classification
of states in Markov chains, compare e.g., \cite{M94, S99}
for the proofs.

\begin{theorem}
\label{mcrectrans}Let $(X_{n})_{n\in\mathbb{N}}$ be a Markov chain
on the finite state space $S$.

\begin{enumerate}
\item A state $j\in S$ is recurrent iff
\[
\underset{n\in\mathbb{N}}{\overset{}{\sum}}p_{n}(j,j)=\infty\text{.}%
\]

\item If the state $j\in S$ is recurrent then
\[
\mathbb{P}\{N_{j}=\infty\mid X_{0}=j\}=1\text{.}%
\]

\item A state $j\in S$ is transient iff
\[
\underset{n\in\mathbb{N}}{\overset{}{\sum}}p_{n}(j,j)<\infty\text{.}%
\]

\item If the state $j\in S$ is transient then it holds for any $i\in S$
\[
\mathbb{P}\{N_{j}<\infty\mid X_{0}=i\}=1\text{ \ \ and}%
\]%
\[
\mathbb{E}(N_{j}\mid X_{0}=i)<\infty\text{.}%
\]

\item If the state $j\in S$ is recurrent and periodic of period $\delta$ then%
\[
\lim_{n\rightarrow\infty}p_{n\delta}(j,j)=\frac{\delta}{\mu(j,j)}\text{,}%
\]
in particular, for aperiodic states we have
\[
\lim_{n\rightarrow\infty}p_{n}(j,j)=\frac{1}{\mu(j,j)}\text{.}%
\]

\item If the state $j\in S$ is positive recurrent and aperiodic, then for
$i\in S$ arbitrary we have%
\[
\underset{n\rightarrow\infty}{\lim}p_{n}(i,j)=\frac{F(i,j)}{\mu(j,j)}\text{.}%
\]

\end{enumerate}
\end{theorem}

\noindent For irreducible chains the qualitative behavior is uniform
for all points, leading to the following results:

\begin{theorem}
\label{mcconverge}Let $(X_{n})_{n\in\mathbb{N}}$ be an irreducible
Markov chain on the finite state space $S$. We fix $j\in S$, then
for all $i\in S$ it holds that

\begin{enumerate}
\item If $j$ has period $\delta$, then so has $i$.

\item If $j$ is transient (recurrent, positive recurrent), then so is $i$. In
fact, all states are either transient or positive recurrent.

\item $\mathbb{P}\{\underset{m\rightarrow\infty}{\lim}\frac{1}{m}%
N_{j}(m)=\frac{1}{\mu(j,j)}\mid X_{0}=i\}=1$.

\item $\underset{n\rightarrow\infty}{\lim}\frac{1}{n}\underset{k=1}%
{\overset{n}{\sum}}p_{k}(i,j)=\frac{1}{\mu(j,j)}$.

\item If $j$ is periodic of period $\delta$ then $\underset{n\rightarrow
\infty}{\lim}p_{n\delta}(i,j)=\frac{\delta}{\mu(j,j)}$ and
$\underset
{n\rightarrow\infty}{\lim}\frac{1}{\delta}\underset{k=n}{\overset{n+\delta
-1}{\sum}}p_{k}(i,j)=\frac{1}{\mu(j,j)}$, in particular for
aperiodic states
$j$ we have $\underset{n\rightarrow\infty}{\lim}p_{n}(i,j)=\frac{1}{\mu(j,j)}%
$. In all cases convergence is geometric with rate $r<1$, where
$r=\max \{\left\vert \lambda\right\vert $, $\lambda$ is an
eigenvalue of $P$ with $\left\vert \lambda\right\vert <1\}$, i.e.
the ergodicity coefficient of the transition matrix $P$.
\end{enumerate}
\end{theorem}

\noindent For irreducible Markov chains the long term behavior for
$n\rightarrow\infty$ is described by ergodicity and stationarity.
Both types of behavior can be formulated using invariant measures of
the chain.

\begin{definition}
Let $(X_{n})_{n\in\mathbb{N}}$ be a Markov chain on the state space
$S$. A probability distribution $\pi^{\ast}$ on $S$ is called
invariant for the chain if $\pi^{\ast}=\mathcal{D}(X_{n})$ for all
$n\in\mathbb{N}$. Here $\mathcal{D}(\cdot)$ denotes again the
distribution of a random variable.
\end{definition}

\begin{remark}
If $(X_{n})_{n\in\mathbb{N}}$ is a Markov chain with transition
probability matrix $P$ on the finite state space $S=\{1,...,d\}$,
then the distribution $\pi^{\ast}$ on $S$ is invariant iff
$\pi^{\ast}\simeq(\pi_{k}^{\ast}$, $k=1,...,d)\in\mathbb{R}^{d}$
satisfies $\pi^{\ast T}=\pi^{\ast T}P$, i.e. if $\pi^{\ast}$ is a
left eigenvector of $P$ corresponding to the (real) eigenvalue $1$.
For irreducible chains this eigenvalue is a simple root of the
characteristic polynomial of $P$, and it is the only one with
absolute value equal to $1$.
\end{remark}

\begin{remark}
\label{mcstatsol}Note that a Markov chain $(X_{n})_{n\in\mathbb{N}}$
is (strictly) stationary (i.e., all its finite-dimensional
distributions are invariant under time shift) iff its initial
variable $X_{0}$ has distribution $\mathcal{D}(X_{0})=\pi^{\ast}$
for some invariant distribution $\pi^{\ast}$.
\end{remark}

\begin{theorem}
\label{mcinvarmeasure}Let $(X_{n})_{n\in\mathbb{N}}$ be an
irreducible Markov chain on the finite state space $S$ with
transition probability matrix $P$.

\begin{enumerate}
\item The Markov chain has a unique invariant distribution $\pi^{\ast}$\ on
$S$.

\item The invariant distribution $\pi^{\ast}\simeq(\pi_{k}^{\ast}$,
$k=1,...,d)\in\mathbb{R}^{d}$ satisfies
$\pi_{k}^{\ast}=\frac{1}{\mu(k,k)} $, where $\mu(k,k)$ denotes again
the mean first return (hitting) time from $k$ to $k$.

\item In particular, all states are positive recurrent.
\end{enumerate}
\end{theorem}

\subsection{Markov Chains, Graphs, and Semiflows}

Usually, presentations about finite state Markov chains develop the
theory for irreducible chains, i.e. the chain consists of one
communicating class. According to Theorems \ref{mcconverge} and
\ref{mcinvarmeasure}, the states of an irreducible chain behave
uniformly in their limit behavior and thus no coexistence of
transient and (positive) recurrent states is possible. We are
interested in studying the qualitative behavior of general finite
state Markov chains. Chapter 9.8 of \cite{R06} presents some results
from a matrix point-of-view. We will do so by using the results from
Section 3. In this section, we develop the mechanisms
that allow us to translate many of those results to the context of
Markov chains.

\noindent Let $(X_{n})_{n\in\mathbb{N}}$ be a Markov chain on the
finite state space $S$ with transition probability matrix $P$. We
associate with $P$ a weighted graph $G=(V,E,w)$, where $V=S$,
$(i,j)\in E$ iff $p(i,j)>0$, and $w:E\rightarrow\lbrack0,1]$ defined
by $w((i,j))=p(i,j)$. The adjacency matrix $A_{G}$ is defined as in
Definition \ref{defadjmatrix} for the graph $G=(V,E)$. Note that the
graph $G=(V,E)$ is automatically an $L-$graph. Vice versa, let
$G=(V,E)$ be a graph with a weight function $w:E\rightarrow
\mathbb{R}$ that satisfies the properties (i)
$w:E\rightarrow\lbrack0,1]$ and (ii) $\underset{j\in
V}{\sum}w((i,j))=1$ for all $i\in V$, then $G=(V,E,w)$ can be
identified with the probability transition matrix $P$ of a Markov
chain on the state space $V$\ via $p(i,j)=w((i,j))$.

\noindent With this construction, we can try to interpret all
concepts and results from Section \ref{orbits&cc} in the context of
Markov chains. Basically, the key ingredient in all proofs is the
following simple observation: Let $\gamma=\left\langle
i_{0},...,i_{n}\right\rangle $ be a path in $G=(V,E)$, then the
probability $\mathbb{P}(\gamma)$\ that this path occurs as a (finite
length) trajectory of the chain $(X_{n})_{n\in\mathbb{N}}$ with
transition probability matrix $P$ and initial distribution
$\pi_{0}$\ is given
by the joint probability%
\begin{align}
\mathbb{P}(\gamma)  & =\mathbb{P}\{X_{0}=i_{0},X_{1}=i_{1},...,X_{n-1}%
=i_{n-1},X_{n}=i_{n}\}\label{probpath}\\
& =\pi_{0}(i_{0})p(i_{0},i_{1})\cdot\cdot\cdot p(i_{n-2},i_{n-1}%
)p(i_{n-1},i_{n})\text{.}\nonumber
\end{align}
In particular, we obtain for any finite sequence $i_{0},...,i_{n}$
of vertices in $G$: $\left\langle i_{0},...,i_{n}\right\rangle $ is
a path in $G $ iff $p(i_{\alpha},i_{\alpha+1})>0$ for
$\alpha=0,...,n-1$.

\noindent This observation implies a probabilistic argument that we
will need for our results in the next section. It is closely related
to the no-cycle property of points for semiflows, compare Definition
\ref{defnocycle}.

\begin{lemma}
\label{leaking}Let $P$ be the transition probability matrix of a
Markov chain on the state space $S$, and let $G=(V,E)$ be its
associated graph. Consider a set $B\subset V$ and a point $i\in B$
with the properties

\begin{enumerate}
\item There exists a path $\gamma\in\Gamma^{n}$ with $\pi_{0}(\gamma)=i$ and
$\pi_{n}(\gamma)=:k\notin B$ for some $n\geq1$,

\item $\mathcal{O}^{+}(k)\cap B=\varnothing$.
\end{enumerate}

\noindent Then $\underset{n\rightarrow\infty}{\lim}p_{n}(i,j)=0$ for
all $j\in B$ uniformly at a geometric rate.
\end{lemma}

\begin{proof}
Denote $\mathbb{P}(\gamma)=\rho$, then\ by (\ref{probpath}) we have
$p_{n}(i,k)\geq\rho$ and hence
\[
\underset{j\in B}{\sum}p_{n}(i,j)\leq1-\rho\text{.}%
\]
Using the Chapman-Kolmogorov equation we
compute for all $j\in B$
\begin{align*}
p_{2n}(i,j)  & =\underset{l\in S}{\sum}p_{n}(i,l)p_{n}(l,j)\\
& =\underset{l\neq k}{\sum}p_{n}(i,l)p_{n}(l,j)+p_{n}(i,k)p_{n}(k,j)\\
& \leq\underset{l\neq k}{\sum}p_{n}(i,l)\underset{l\neq k}{\sum}p_{n}(l,j)+0\\
& \leq(1-\rho)\underset{l\neq k}{\sum}p_{n}(l,j)\\
& \leq(1-\rho)^{2}\text{.}%
\end{align*}
Repeating this argument we obtain for all $m\geq1$%
\[
p_{mn}(i,j)\leq(1-\rho)^{m}\text{.}%
\]
By assumption 2 of the lemma, we see that
$p_{mn+\alpha}(i,j)\leq(1-\rho)^{m}$ for $0\leq\alpha\leq n-1$,
which proves the assertion. Note that the argument above even shows
$\underset{n\rightarrow\infty}{\lim}\underset{j\in B}{\sum
}p_{n}(i,j)=0$ at the geometric rate $(1-\rho)$.
\end{proof}

\noindent Next we comment briefly on the connection between Markov
chains and semiflows defined by products of matrices, i.e. {\bf linear
iterated function systems}. The standard connection is constructed as follows: Let $(X_{n}%
)_{n\in\mathbb{N}}$ be a Markov chain on the finite state space
$S=\{1,...d\}$ with transition probability matrix $P$. We identify
the probability measures on $S$ with the set of probability vectors
in $\mathbb{R}^{d}$, defined as
$\mathcal V=\{v\in\mathbb{R}^{d}$, $v_{i}\geq0$ for all $i=1,...,d$ and $%
{\textstyle\sum}
v_{i}=1\}$. Then $\Upsilon:\mathbb{N}\times \mathcal V\rightarrow \mathcal V$, defined
by $\Upsilon(n,v)=(v^{T}P^{n})^{T}$ is a semiflow. If
$v_{0}=\mathcal{D}(X_{0})$, this semiflow describes the evolution of
the $1-$dimensional (and the $n-$dimensional) distributions of the
Markov chain, thanks to the Chapman-Kolmogorov equation. Vice versa, if $\Upsilon$ is a linear,
iterated function system on $\mathbb{R}^{d}$ that leaves $\mathcal V$
invariant, then $\Upsilon$ can be interpreted as a Markov chain on
the state space $S=\{1,...,d\}$. In standard texts, many results
about finite state Markov chains are proved using this connection.

\noindent Our discussion in Section \ref{matrix} suggests another
matrix semiflow associated to a Markov chain, namely the semiflow
$\Psi :\mathbb{N}\times\mathcal{Q}^{d}\rightarrow\mathcal{Q}^{d}$,
$\Psi
(n,q)=(q^{T}\cdot^{\ast}A^{n\ast})^{T}$ on the unit cube $\mathcal{Q}%
^{d}\subset\mathbb{R}^{d}$, compare Definition \ref{defmatrixflow}.
Here $A\in M(d,\{0,1\})$ is defined by $a_{ij}=1$ if $p_{ij}>0$, and
$a_{ij}=0$ otherwise. This semiflow does not propagate the
distributions of the Markov chain, just its reachability or
$\{0,1\}-$structure. It follows from Section \ref{matrix} that all
concepts and results for $L-$graphs can be interpreted in terms of
this semiflow, e.g. Equation (\ref{probpath}) and Lemma
\ref{leaking} have obvious translations to the context of $\Psi$.

\noindent Finally we consider semiflows on power sets: Let $(X_{n}%
)_{n\in\mathbb{N}}$ be a Markov chain on the finite state space
$S=\{1,...d\}$ with transition probability matrix $P$. Let 
$\mathcal{P}(S)$ be the power set of $S$, define the semiflow
$\Phi:\mathbb{N}\times\mathcal{P}(S)\rightarrow\mathcal{P}(S)$,
through $\Phi(n,A)=\{j\in S$, there exists $i\in A$ such that $p_{n}%
(i,j)>0\}$. The flow $\Phi$ is of the type (\ref{defflow+}) and
hence it satisfies all the properties studied in Sections
\ref{graphsemiflow}-\ref{arsemiflow}. In particular, Equation
(\ref{probpath}) and Lemma \ref{leaking} have obvious translations
to the context of $\Phi$.

\noindent Note that, in contrast to the semiflow $\Upsilon$ (and the
weighted graph $(V,E,w)$), the semiflows $\Psi$ and $\Phi$ (and the
graph $G=(V,E)$) do not, by definition, carry all the statistical
(or distributional) information of the Markov chain
$(X_{n})_{n\in\mathbb{N}}$, just the information about certain basic
events (such as paths) occurring with probability zero or with
positive probability. Hence when using this graph or these semiflows
to analyze the Markov chain, we can only hope for some statements of
the kind "an event defined by the Markov chain has positive
probability or probability $0$". We will see in the next section
that the graph and the semiflows do characterize a surprisingly wide
array of properties of the Markov chain.

\subsection{Characterization of Markov Chains via Graphs and
Semiflows\label{dictionary}}

In this section let $(X_{n})_{n\in\mathbb{N}}$ be a Markov chain on
the finite state space $S=\{1,...d\}$ with transition probability
matrix $P$. When we talk about probability measures on $S$ we always
think of $S$ as endowed with the discrete $\sigma-$algebra
$\mathcal{P}(S)$. Associated with
$(X_{n})_{n\in\mathbb{N}}$ are the graph $G=(S,E)$ and the semiflows
$\Phi:\mathbb{N}\times\mathcal{P}(S)\rightarrow\mathcal{P}(S)$ and $\Psi
:\mathbb{N}\times\mathcal{Q}^{d}\rightarrow\mathcal{Q}^{d}$, as
defined in the previous section. The goal of this section is to
develop a ``three-column dictionary" of equivalent objects and
results in the three languages of Markov chains, graphs, and
semiflows. To avoid confusion between the semiflows
$\Upsilon:\mathbb{N}\times \mathcal V\rightarrow \mathcal V$, defined on the set of
probability vectors in $\mathbb{R}^{d}$ and containing all the
probabilistic information of the chain, and
$\Psi:\mathbb{N}\times\mathcal{Q}^{d}\rightarrow \mathcal{Q}^{d}$,
defined on the unit cube in $\mathbb{R}^{d} $ and containing only
the reachability information of the chain, we will formulate our
observations in terms of the semiflow $\Phi$. The results
immediately carry over to $\Psi$ using the correspondence from
Section \ref{matrix}.

\subsubsection{Paths, Orbits, Supports of Transition Probabilities, and First
Hitting Times}

For a probability measure $\mu$ on $(S,\mathcal P(S))$ the support
supp$\mu$ is defined as the smallest subset $S^{\prime}\subset S$
such that $\mu(S^{\prime })=1$. Note that the $n-$step transition
probabilities $p_{n}(\cdot,\cdot)$
define probability measures $P(n,i,\cdot)$ on $S$ via $P(n,i,A):=%
{\textstyle\sum\nolimits_{j\in A}}
p_{n}(i,j)$. That is, each row of the matrix $P^{n}$ ``is" a probability
measure for all $n\in\mathbb{N}$. \vspace{0.25in}

\noindent\textbf{Fact 1:} A finite sequence of points
$(i_{0},...,i_{n})$ in $S$ is a path of $G$ iff
$p_{n}(i_{0},i_{n})>0$ iff $i_{n}\in\Phi (n,\{i_{0}\})$. Each of
these statements is equivalent to $i_{n}\in$ supp$P(n,i_{0},\cdot)$.

\noindent The proof of this fact follows directly from
(\ref{probpath}). An immediate consequence of this fact
is:\vspace{0.25in}

\noindent\textbf{Fact 2:} Fix a point $i\in S$. Then for $j\in S$ we
have: $j\in\mathcal{O}^{+}(i)$ iff $j\in$ supp$P(n,i,\cdot)$ for
some $n\geq1$ iff
$j\in%
{\textstyle\bigcup\nolimits_{n\geq1}}
\Phi(n,\{i\})$.\vspace{0.25in}

\noindent\textbf{Fact 3:} A subset $A\subset S$ is forward invariant
for $G$ iff $A$ is stochastically closed for the Markov chain iff
$\Phi(n,A)\subset A$ for all $n\geq1$.

\noindent The proof of this fact follows directly from Fact 2. This
fact also allows us to characterize certain properties of first
hitting times for Markov chains.\vspace{0.25in}

\noindent\textbf{Fact 4:} Let $i\in S$ and $A\subset S$. The first
hitting time distribution $(f_{n}(i,A)$, $n\geq1)$ satisfies for any
$n\in\mathbb{N}$: $f_{n}(i,A)>0$ iff there exists a path
$\gamma\in\Gamma^{n}$ with $\pi _{0}(\gamma)=i$, $\pi_{n}(\gamma)\in
A$, and $\pi_{m}(\gamma)\notin A$ for all $m=1,...,n-1$.
Furthermore, $F(i,A)>0$ iff $\mathcal{O}^{+}(i)\cap
A\neq\varnothing$.

\noindent This result follows immediately from (\ref{probpath}) and
Fact 2. The definition of the semiflow $\Phi$ in Equation
(\ref{defflow+}) allows directly for an equivalent statement in
terms of $\Phi$. Note that we did not define ``first hitting times"
for graphs or semiflows because this concept is hardly, if ever,
used in graph theory. However, if we define for the graph $G=(S,E)$
the notion $\sigma_{A}(i):=\inf\{n\geq1$, there exists $\gamma
\in\Gamma^{n}$ with $\pi_{0}(\gamma)=i$ and $\pi_{n}(\gamma)\in A\}$
as the first hitting time of a set $A\subset S$ for paths starting
in $i\in S$, then we obviously have from Fact 4 that
$\sigma_{A}(i)=n$ implies $f_{n}(i,A)>0$, but the converse is, in
general, not true.

\noindent As a final observation for this section, we note that
Definition \ref{defmcperiod} of the period of a state $i\in S$ only
uses the property of $p_{n}(i,i)>0$. Hence according to Fact 1, this
is really a pathwise property of the graph $G$: Consider the graph
$G=(S,E)$ and a vertex $i\in S$, we define the period of $i$ as
$\eta(i)=\gcd\{n\geq1$, there exists a path $\gamma\in\Gamma^{n}$
with $\pi_{0}(\gamma)=i$ and $\pi_{n}(\gamma)=i\}$. Then, by Fact 1,
we have $\eta(i)=\delta(i)$, the period of $i$ as a state of the
Markov chain. It seems, however, that periods of vertices are
rarely, if ever, used in graph theory.


\subsubsection{Communication and Communicating Classes}\label{nonref}

We have defined communicating classes for graphs in Definition
\ref{defcc} and for Markov chains after Definition
\ref{defmccommun}. These definitions differ slightly: For graphs we
required paths of length $n\geq1$, while we followed standard
practice for Markov chains and allowed $n=0$. The Markov chain
practice renders the communication relation $\sim$ an equivalence
relation, but it also leads to trivial communicating classes of the
form $\{i\}$ with $p_{n}(i,i)=0$ for all $n\geq1$. In the context of
graphs we called such vertices transitory, compare Definition
\ref{deftransitory}. Hence we obtain the following
results.\vspace{0.25in}

\noindent\textbf{Fact 5:} Consider two points $i,j\in S$. $i$ and
$j$ communicate via the graph $G$ iff $i$ and $j$ communicate via
the Markov chain and $p_{n}(i,j)>0$ and $p_{m}(j,i)>0$ for some
$n,m\geq1$.\vspace{0.25in}

\noindent\textbf{Fact 6:} A point $i\in S$ is a transitory vertex of
the the graph $G$ iff $p_{n}(i,i)=0$ for all $n\geq1$.
\vspace{0.25in}

\noindent\textbf{Fact 7:} A set $A\subset S$ is weakly invariant for
the semiflow $\Phi$ iff for all $n\in\mathbb{N}$ there exist
$i_{n},j_{n}\in A$ with $p_{n}(i_{n},j_{n})>0$. \vspace{0.25in}

\noindent\textbf{Fact 8:} A subset $C\subset S$ is a communicating
class of the graph $G$ iff $C$ is a nontrivial communicating class
of the Markov chain iff $C$ is a finest Morse set of the semiflow
$\Phi$.\vspace{0.25in}

\noindent\textbf{Fact 9:} A subset $C\subset S$ is a maximal
communicating class of the graph $G$ iff $C$ is a stochastically
closed communicating class of the Markov chain iff $C$ is a minimal
(with respect to set inclusion) attractor of $\Phi$.\vspace{0.25in}

\noindent\textbf{Fact 10:} A point $i\in S$ satisfies $i\in C$ for
some communicating class $C$ of the graph $G$ iff there exists a
sequence $n_{k}\rightarrow\infty$ with $p_{n_{k}}(i,i)>0$ iff
$i\in\omega(A)$, the $\omega-$limit set of some $A\subset S$ under
the semiflow $\Phi$ iff $\{i\}$ is a recurrent point of $\Phi$.

\noindent As we will see below, recurrence for semiflows (as defined
in Definition \ref{defrecur+}) and for Markov chains (as defined in
Definition \ref{defrecurmc}) are two different concepts, as are the
concepts of transitory points of semiflows and transient points of
Markov chains.

\subsubsection{Recurrence, Transience, and Invariant Measures for Markov
Chains\label{rtimformc}}

Let $M\subset S$ be a stochastically closed set for Markov chain
$(X_{n})_{n\in\mathbb{N}}$ with transition probability matrix $P$.
The
restriction of $P$ to $M$ defines a new Markov chain $(X_{n}^{M}%
)_{n\in\mathbb{N}}$, whose transition probability matrix we denote
by $P^{M}$. Since $M$ is stochastically closed, we have for all
$i,j\in M$ that $p_{n}(i,j)=p_{n}^{M}(i,j)$ holds for all
$n\in\mathbb{N}$. Therefore we conclude from Theorem
\ref{mcrectrans}: $i\in M$ is a transient (recurrent, positive
recurrent) point for $X$ iff it is transient (recurrent, positive
recurrent) for the restricted chain $X^{M}$. And Theorems \ref{mcconverge}%
\ and \ref{mcinvarmeasure} imply that $\mu$ is an invariant
probability measure of $X$ with supp$\mu\cap M\neq\varnothing$ iff
$\mu$ is an invariant probability measure for $X^{M}$. With these
preparations we can characterize the long term behavior of states in
a Markov chain:\vspace{0.25in}

\noindent\textbf{Fact 11:} A point $i\in S$ is a transient point of
the Markov chain $(X_{n})_{n\in\mathbb{N}}$ iff $i\in
S\backslash\cup C$, $C$ a maximal communicating class of the graph
$G=(S,E)$.

\begin{proof}
Maximal communicating classes $C$ of $G$ are forward invariant by
Remark \ref{graphsum}. Hence the Markov chain $X^{C}$ is irreducible
and all states $i\in C$ are (positive) recurrent by Theorem
\ref{mcinvarmeasure}.3. From the observation above we then have that
$i$ is (positive) recurrent for $X$, which proves the $\Rightarrow$
direction. Vice versa, if $i\in S\backslash\cup C$, then Lemma
\ref{leaking} with $B=\cup C$ implies $\underset{n\rightarrow
\infty}{\lim}p_{n}(i,i)=0$ at a geometric rate. Hence $\underset
{n\in\mathbb{N}}{\sum}p_{n}(i,i)<\infty$ and therefore the state $i$
is transient by Theorem \ref{mcrectrans}.3.
\end{proof}

\noindent Note that the transitory points for the graph $G=(S,E)$
are a subset of the transient states of the Markov chain: Only those
points are transitory that are not element of any communicating
class.\vspace{0.25in}

\noindent\textbf{Fact 12:} A point $i\in S$ is a recurrent (and
positive recurrent) point of the Markov chain
$(X_{n})_{n\in\mathbb{N}}$ iff $i\in\cup C$, $C$ a maximal
communicating class of the graph $G=(S,E)$.

\noindent The proof of this fact is immediate from Fact 11 and
Theorems \ref{mcrectrans} and \ref{mcinvarmeasure}.3. Note that the
semiflow concept of recurrence is different from that for Markov
chains: Under the semiflow $\Phi$ the points in any communicating
class are exactly the $\Phi-$recurrent points.

\noindent We now turn to invariant distributions of the Markov chain
$(X_{n})_{n\in\mathbb{N}}$. If $\mu_{1}$ and $\mu_{2}$ are two
invariant probability measures of a Markov chain, then any convex
combination $\mu=\alpha\mu_{1}+(1-\alpha)\mu_{2}$ for
$\alpha\in\lbrack0,1]$ is obviously again an invariant probability
measure. An invariant probability measure $\mu$ is called extreme if
it cannot be written as a convex combination $\mu
=\alpha\mu_{1}+(1-\alpha)\mu_{2}$ of two different invariant
probability measures $\mu_{1}\neq\mu_{2}$ for
$\alpha\in(0,1)$.\vspace{0.25in}

\noindent\textbf{Fact 13:} Each maximal communicating class
$C_{\nu}$ of the graph $G=(S,E)$ is the support of exactly one
invariant distribution $\mu _{\nu}$ of the Markov chain
$(X_{n})_{n\in\mathbb{N}}$. The chain has exactly $l$ extreme
invariant distributions $\mu_{1},...\mu_{l}$, where $l$ is the
number of maximal communicating classes of the graph $G$. Any
invariant distribution $\mu$ of the chain is a convex combination
$\mu=\underset{\nu =1}{\overset{l}{\sum}}\alpha_{\nu}\mu_{\nu}$ with
$\sum\alpha_{\nu}=1$.

\begin{proof}
Maximal communicating classes $C_{\nu}$ of $G$ are stochastically
closed, hence the Markov chain restricted to $C_{\nu}$\ is
irreducible and therefore $C_{\nu}$\ is the support of a unique
invariant distribution $\mu_{\nu}$ by Theorem
\ref{mcinvarmeasure}.1. Now the other claims follow directly from
Fact 9. \vspace{0.25in}
\end{proof}

\noindent\textbf{Fact 14:} The Markov chain
$(X_{n})_{n\in\mathbb{N}}$ has a
unique (in the distributional sense) stationary solution $(X_{n}^{\nu}%
)_{n\in\mathbb{N}}$ on each maximal communicating class $C_{\nu}$ of
the graph $G=(S,E)$, obtained by taking initial random variables
$X_{0}$ with distribution $\mathcal{D}(X_{0})=\mu_{\nu}$. All other
stationary solutions are convex combinations of the $X^{\nu}$,
$\nu=1,...,l$.

\noindent This fact is a combination of Remark \ref{mcstatsol} and
Fact 13.

\subsubsection{Global Behavior and Multistability}

This section is devoted to the convergence behavior of a Markov
chain $(X_{n})_{n\in\mathbb{N}}$ as $n\rightarrow\infty$. We have
already characterized the transient and (positive) recurrent points.
It remains to clarify how the chain behaves starting in any of these
points. When one deals with difference (or differential) equations,
one first looks for fixed points (and other simple limit sets) and
then tries to find the initial values from which the system
converges towards these fixed points (or more generally, limit
sets). The study of the global behavior of dynamical systems
clarifies these issues, compare, e.g., \cite{AH06} for the case of
flows, and Section 3 for an adaptation to
specific semiflows related to $L-$graphs and to Markov chains.

\noindent Markov chains, in general, do not have fixed points.
According to Fact 10, their long term behavior is determined by the
communicating classes. And Fact 12 suggests that the maximal
communicating classes determine the behavior of a Markov chain
$(X_{n})_{n\in\mathbb{N}}$ as $n\rightarrow\infty$. The following
facts show that this is, indeed, true. The special case of a fixed
point is recovered when a maximal communicating class consists of
exactly one state; such states are called absorbing.\vspace{0.25in}

\noindent\textbf{Fact 15:} Consider the Markov chain $(X_{n})_{n\in\mathbb{N}%
}$ with transition probability matrix $P$, initial random variable
$X_{0}$ and initial distribution $\mathcal{D}(X_{0})=\pi_{0}$. Let
$C$ be a maximal communicating class of the graph $G=(S,E)$ with
invariant probability $\mu^{\ast}$. Assume that $\pi_{0}$ is
concentrated on $C$, i.e. $\pi_{0}(i)>0
$ iff $i\in C$. Then $$\underset{n\rightarrow\infty}{\lim}\frac{1}{n}%
\underset{k=1}{\overset{n}{\sum}}\mathcal{D}(X_{k})=\mu^{\ast}$$ in
the distributional sense, i.e. as vectors in $\mathbb{R}^{d}$. If,
furthermore, $C$ has period $\delta$ then
$$\underset{n\rightarrow\infty}{\lim}\frac
{1}{\delta}\underset{k=n}{\overset{n+\delta-1}{\sum}}\mathcal{D}(X_{k}%
)=\mu^{\ast}$$

\noindent In particular for $C$ aperiodic we have $\underset
{n\rightarrow\infty}{\lim}\mathcal{D}(X_{n})=\mu^{\ast}$. All
convergences are at a geometric rate. This fact follows immediately from Theorems
\ref{mcconverge}.5 and \ref{mcinvarmeasure}.2, and the remarks at
the beginning of Section \ref{rtimformc}.\vspace{0.25in}

\noindent\textbf{Fact 16:} Consider the Markov chain $(X_{n})_{n\in\mathbb{N}%
}$. Let $C=\cup C_{\nu}$, $C_{\nu}$ a maximal communicating class of
the graph $G=(S,E)$, and $D:=S\backslash C$. Pick $i\in D$. Then the
first hitting time $\tau_{C}$ satisfies
$\mathbb{P}\{\tau_{C}<\infty\mid X_{0}=i\}=1$ and
$\mathbb{E}(\tau_{C})<\infty$.

\begin{proof}
Assume, under the given assumptions, that
$\mathbb{P}\{\tau_{C}<\infty\mid X_{0}=i\}<1$, i.e.
$\mathbb{P}\{\tau_{C}=\infty\mid X_{0}=i\}>0$. Then there exists a
state $j\in S\backslash C$ with $P\{N_{j}=\infty\mid X_{0}=j\}>0$.
By the characterization from Theorem \ref{mcrectrans}.4 the state
$j$ cannot be transient, and hence by the dichotomy in Theorem
\ref{mcconverge}.2 $j$ is recurrent, which contradicts Fact 12.
Hence $\mathbb{P}\{\tau_{C}<\infty\mid X_{0}=i\}=1$ and then
$\mathbb{E}(\tau_{C})<\infty$ follows from geometric convergence in
Lemma \ref{leaking}.
\end{proof}

\noindent We define the probability that the chain, starting in
$D=S\backslash
C$, hits the maximal communicating class $C_{\nu}$ by $p_{\nu}:=\mathbb{P}%
\{X_{\tau_{C}}\in C_{\nu}\mid X_{0}\in D\}$. Note that $\sum
p_{\nu}=1$. Then the long-term behavior of
$(X_{n})_{n\in\mathbb{N}}$ with $\mathcal{D}(X_{0})$ concentrated in
$D$ is as follows: $X_{n}$ will leave the set $D$ of transient
states in finite time (even with finite expectation) and enter into
the set $C$ of (positive) recurrent points, where it may enter one
or more of the $C_{\nu}$'s depending on whether $p_{\nu}$ is
positive or not. By the random version of the Chapman-Kolmogorov
equation (\ref{randomCKequation}), the process continues as a Markov
chain, and hence in each $C_{\nu}$ follows the behavior described in
Fact 15. In particular, if all $C_{\nu}$ are aperiodic, we
obtain:\vspace{0.25in}

\noindent\textbf{Fact 17:} Consider the Markov chain $(X_{n})_{n\in\mathbb{N}%
}$ with maximal communicating classes $C_{1},...,C_{l}$ and extreme
invariant measures $\mu_{1},...,\mu_{l}$. Assume that all $C_{\nu}$,
$\nu=1,...,l$ are aperiodic. Then we have as limit behavior of the
chain $$\underset
{n\rightarrow\infty}{\lim}\mathcal{D}(X_{n})=\overset{l}{\underset{\nu=1}%
{\sum}}p_{\nu}\mu_{\nu}$$

 \noindent where the $p_{\nu}$ are as defined above.

\noindent This fact leads to the definition of multistable states: A
state $i\in S$ is called multistable for the Markov chain
$(X_{n})_{n\in\mathbb{N}} $ if there exist maximal communicating
classes $C_{1}$ and $C_{2}$ with $C_{1}\neq C_{2}$ such that
$p_{\nu}>0$ for $\nu=1,2$. We now obtain the following fact as a
criterion for the existence of multistable states. 

%

\noindent\textbf{Fact 18:} Let $(X_{n})_{n\in\mathbb{N}}$ be a
Markov chain with connected graph $G$. Then the chain has
multistable states iff $G$ has at least two maximal communicating
classes.


\noindent Multistable, and specifically bistable states play an
important role in many applications of stochastic processes in the
natural sciences and in engineering see \cite{CRK96}.

{}

\end{document}